\documentclass[11pt]{article}

\usepackage[T1]{fontenc}
\usepackage{verbatim}
\usepackage{amsmath,amsthm}
\usepackage{xspace}
\usepackage{pifont}
\usepackage{amssymb}
\usepackage{dsfont}
\usepackage{makeidx}
\usepackage{mathrsfs}
\usepackage{xcolor, colortbl}
\usepackage{subcaption}
\usepackage[numbers]{natbib}
\usepackage{enumitem}
\usepackage{afterpage}
\usepackage{mathtools}
\usepackage{graphics}
\usepackage[export]{adjustbox}
\usepackage{tcolorbox}
\usepackage{authblk}
\usepackage[normalem]{ulem}
\usepackage[colorlinks,citecolor=blue,urlcolor=blue]{hyperref}
\mathtoolsset{showonlyrefs}

\def\eps{\varepsilon}

\def\Q{\mathbb{Q}}
\def\P{\mathbb{P}}
\def\E{{\mathbb{E}}}
\newcommand{\R}{{\mathbb R}}
\newcommand{\F}{{\mathbb F}}
\newcommand{\N}{{\mathbb N}}

\DeclareMathOperator*{\argmax}{arg\,max}

\newcommand{\bd}{\begin{displaymath}}
\newcommand{\ed}{\end{displaymath}}
\newcommand{\be}{\begin{equation}}
\newcommand{\ee}{\end{equation}}
\newcommand{\bq}{\begin{eqnarray}}
\newcommand{\eq}{\end{eqnarray}}
\newcommand{\bn}{\begin{eqnarray*}}
\newcommand{\en}{\end{eqnarray*}}

\newtheorem{theorem}{Theorem}[section]
\newtheorem{lemma}[theorem]{Lemma}
\newtheorem{proposition}[theorem]{Proposition}

\newtheorem{convention}[theorem]{Convention}
\newtheorem{corollary}[theorem]{Corollary}

\newtheorem{remark}[theorem]{Remark}

\newtheorem{example}[theorem]{Example}

\newtheorem{definition}[theorem]{Definition}
\newtheorem{assumption}[theorem]{Assumption}
\numberwithin{equation}{section}

\date{\today}
\title{Stochastic Games on Large Sparse Graphs}
\author[]{Eyal Neuman}
\author[]{Sturmius Tuschmann\thanks{ST is supported by the EPSRC Centre for Doctoral Training in Mathematics of Random \mbox{Systems}: Analysis, Modelling and Simulation (EP/S023925/1).}}
\affil[]{Department of Mathematics, Imperial College London}

\begin{document}

\maketitle

\begin{abstract}
We introduce a framework for stochastic games on large sparse graphs, covering continuous-time and discrete-time dynamic games as well as static games. Players are indexed by the vertices of simple, locally finite graphs, allowing both finite and countably infinite populations, with asymptotics described through local weak convergence of marked graphs. The framework allows path-dependent utility functionals that may be heterogeneous across players. Under a contraction condition, we prove existence and uniqueness of Nash equilibria and establish exponential decay of correlations with graph distance. We further show that global equilibria can be approximated by truncated local games, and can even be reconstructed exactly on subgraphs given information on their boundary. Finally, we prove convergence of Nash equilibria along locally weakly convergent graph sequences, including sequences sampled from hyperfinite unimodular random graphs.
\end{abstract}

\begin{description}
\item[Mathematics Subject Classification (2020):] 	91A07, 91A15, 91A43, 93E20 
\item[Keywords:] network games, sparse graphs, random rooted graphs, local weak convergence, Nash equilibrium, unimodularity, hyperfiniteness 
\end{description}

\section{Introduction}\label{sec:introduction}
Dynamic network games are non-cooperative games in which players make strategic decisions within a dynamic system, with their mutual dependencies encoded by a connectivity network modeled as a graph. Each player’s strategy influences not only their own payoff but also the dynamics of the environment. These games are useful for modeling strategic behavior in a range of competitive systems, such as autonomous driving \cite{hang2020human}, real-time bidding \cite{sayedi2018real}, and dynamic production management \cite{leng2005game}. From a mathematical standpoint, dynamic network games are challenging to analyze because interactions are typically heterogeneous, populations can be very large, and the feedback effects between players' actions and the evolving environment are often nonlinear.
 
Lovász and Szegedy's seminal work on the convergence of dense graph sequences paved the way for the development of graphon games, which serve as infinite-population approximations of finite-player network games. They showed that in dense graph sequences (where the individual node degrees tend to infinity as the total number of nodes increases) subgraph densities converge to a natural limit. This limit is represented by a graphon, i.e., a symmetric measurable function $W:[0,1]^2 \to [0,1]$ that encodes these densities (see \citep{borgs2008convergent,borgs2012convergent,lovasz2012large}).

Building on this graphon formalism, \citet{parise2023graphon} introduced static graphon games as an infinite-population approximation for large finite static network games. In these games, a continuum of players interacts according to a graphon $W$, where $W(u,v)$ specifies the strength of interaction between players $u$ and $v$. Crucially, the Nash equilibria of these graphon games provide approximate Nash equilibria for large finite network games \cite{carmona2022stochastic, parise2021analysis}.
Dynamic network games are analytically more involved than static ones, since each player’s payoff depends not only on the others' actions but also on an evolving system state whose dynamics are jointly driven by all actions. Dynamic graphon games with Markovian state dynamics have been studied in \cite{aurell2022stochastic, caines2019graphon,caines2021graphon,gao2021linear,lacker2023label,plank2025policy,tangpi2024optimal}, among others, while non-Markovian formulations were studied in \cite{neuman2024stochastic, neuman2025stochastic}.
 
While graphon models provide a powerful limit theory for dense graphs, most large-scale real-world networks are sparse. Games on sparse networks are thus particularly important, as they more accurately reflect the structure of such real-world networks, where players face cognitive, economic, or technological constraints that limit the number of connections they can maintain. In these settings, the individual node degrees typically remain bounded or grow slowly with network size, making dense models less realistic. As a result, sparse models are often better suited for studying strategic behavior, efficiency, and policy design in large networks arising in economic, social, and technological applications. While the limit theory for dense graphs is well established, the sparse case remains comparatively less understood.

One approach to scaling limits of sparse graph sequences is local weak convergence, introduced by \citet{benjamini2001recurrence} and independently by \citet{aldous2004objective}. For sequences of locally finite graphs, local weak convergence captures the distributional convergence of rooted neighborhoods around a uniformly chosen node. That is, it amounts to the weak convergence of the associated laws on the space of locally finite, connected, countable, rooted graphs equipped with the local topology. A central concept in this theory is that of unimodular probability measures, which play a role analogous to graphons as limit objects for dense graph sequences \cite{Aldous2007, benjamini2015unimodular, bordenave2016lecture}. 

Local weak convergence has recently found applications as a limit theory for locally interacting dynamic systems. In particular, \citep{ganguly2022interacting,ganguly2024hydrodynamic,guo2025particle,LackerRamananWu2021, LackerRamananWu2023b,lacker2023local,lacker2025marginal,oliveira2019interacting,OliveiraReisStolerman2020}, among others, have studied large-scale interacting diffusions and Markov chains on sparse graphs. However, extending this paradigm to stochastic games poses significant challenges, as it integrates dynamic interactions on sparse networks with game-theoretic effects. Some special cases have been studied by \citet{lacker2022case}, who considered linear-quadratic stochastic differential games on the specific class of transitive graphs and established convergence of Markovian Nash equilibria under weak convergence of the associated graph spectral measures. However, as pointed out in \cite{lacker2022case,ramananICM}, a general framework for stochastic games on locally finite graphs, including an intrinsic infinite-player model that allows the approximation of large finite-player models, has been missing. Addressing this challenge is the main focus of this work.

Concretely, we introduce a general formulation for stochastic games on large sparse graphs that accommodates continuous-time and discrete-time dynamic games as well as static ones. Players are indexed by the nodes of a simple, locally finite graph, allowing for both finite and countably infinite populations, while asymptotics are encoded through local weak convergence of marked graphs, where each graph node is equipped with an associated mark that carries the player's control. Our framework permits path-dependent goal functionals that are heterogeneous across players, and may include both idiosyncratic and common noise. More precisely, we work with continuously differentiable, strongly concave utility functionals, which provides a general yet tractable setting in which we can characterize Nash equilibria and study their structural properties.

\paragraph{Our main contributions.} In Theorem~\ref{thm:NE}, we show that for any simple, locally finite graph, the game admits a unique Nash equilibrium, obtained as the fixed point of the best-response operator. Having established existence and uniqueness, we then investigate structural properties of this equilibrium and prove its convergence.

\begin{enumerate}
    \item [\textbf{(i)}] \textbf{Correlation decay.} Assuming that the players’ heterogeneity processes (i.e., their individual noise terms) are independent, Proposition~\ref{prop:correlation-decay} proves that the equilibrium strategies exhibit exponential decay of correlations with graph distance. This result is established through the interplay between the Picard approximation of the Nash equilibrium and the rate at which information propagates through the graph.
    
    \item [\textbf{(ii)}]  \textbf{Local approximation.} We formalize the notion of approximating a player’s equilibrium action using only information from a local neighborhood of that player. Proposition~\ref{prop:local-approximation} makes this precise and yields exponential bounds on the convergence rate. Building on this, Proposition~\ref{prop:epsilon-Nash} constructs global $\varepsilon$-Nash equilibria from local equilibrium approximations and derives exponential bounds on the resulting utility gap relative to the global Nash equilibrium.

    \item [\textbf{(iii)}]  \textbf{Local reconstruction.} Proposition~\ref{prop:local-reconstruction} shows that if the Nash equilibrium is known on the vertex boundary of an induced subgraph, then the equilibrium in the interior can be recovered exactly by iterating the best-response operator while keeping the boundary actions fixed.

    \item [\textbf{(iv)}] \textbf{Convergence.} For a sequence of finite random graphs converging locally weakly to a random rooted graph, Theorem~\ref{thm:convergence} establishes convergence of the associated Nash equilibria. This result requires the convergent graph sequence to be specified in advance and is proved by combining analytic fixed-point arguments with tools from local weak convergence theory. We prove this both for convergence in distribution in the local weak sense and for the stronger notion of convergence in probability in the local weak sense, with the latter requiring refined techniques.

    We also introduce a sampling scheme for a hyperfinite unimodular random graph $G$, producing an increasing sequence of finite random subgraphs $(G_n)_{n\ge 1}$ that converges locally weakly to $G$. In Corollary~\ref{cor:convergence} we show that the Nash equilibria of the games on the sampled graphs converge to the Nash equilibrium on $G$.
\end{enumerate}

Collectively, these results establish a unified framework for analyzing strategic interactions on large sparse networks. Beyond a general existence–uniqueness result, we show that equilibrium behavior is intrinsically local: correlations decay exponentially with graph distance, local computations yield global approximate equilibria, and interior equilibrium actions on subgraphs can be exactly recovered from boundary values. Complementing these locality results, we connect finite and infinite populations by proving equilibrium convergence under local weak convergence, including convergence under graph sampling.

\paragraph{Structure of the paper.} The remainder of the paper is structured as follows. Section~\ref{sec:preliminaries} reviews sparse graphs, their local weak convergence, and the corresponding limit objects. Section~\ref{sec:model} introduces our general framework for stochastic games on large sparse graphs, along with several examples. Section~\ref{sec:results} presents our main results, including existence and uniqueness of equilibria, correlation decay of equilibrium actions, local approximation and local reconstruction of equilibria, and convergence of equilibria. Sections~\ref{sec:proofs-existence-uniqueness}--\ref{sec:proofs-convergence} are dedicated to the proofs of our results. Additional results can be found in the appendix.

\section{Preliminaries}\label{sec:preliminaries}
In this section we collect the basic notions needed to describe sparse graphs and their convergence. In particular, we define local weak convergence of (random) graph sequences, as introduced by \citet{benjamini2001recurrence} and independently by \citet{aldous2004objective}. We also consider the concept of unimodular probability measures, which will serve as limit objects in the spirit of graphons for dense graph sequences. Following \cite{Aldous2007,benjamini2015unimodular,bordenave2016lecture,lacker2023local,van2024random}, we work with both unmarked and marked graphs.

\subsection{Graphs}\label{subsec:graphs}
Throughout, we will work with graphs $G=(V,E)$ that always consist of a finite or countably infinite vertex set $V$ and an edge set $E$. All graphs are assumed to be simple (i.e.,~they are undirected and have no loops or multiple edges) and locally finite (that is, the degree $\deg_G(v)$ of each vertex $v\in V$ is finite). For a vertex $v\in V$, let $N_v(G):=\{u\in V: \{u,v\}\in E\}$ denote the set of all neighbors of $v$ in $G$ and write $u\sim v$ when $u\in N_v(G)$. For two vertices $u,v\in V$, denote by $d_G(u,v)$ the graph distance between $u$ and $v$, that is, the length of the shortest path in $G$ from $u$ to $v$. If $H=(V_H,E_H)$ is a graph with $V_H\subset V$ and $E_H\subset E$, we call $H$ a subgraph of $G$ and write $H\subset G$.
We call $H$ an induced subgraph if $E_H=\{\{u,v\}\in E:\ u,v\in V_H\}$, that is, $H$ contains all edges of $G$ between vertices in $V_H$. Given a subgraph $H\subset G$ and $k\in\N_0$, we define its $k$-neighborhood $B_k(H):=\{u\in V(G): \min_{v\in V_H} d_G(u,v)\le k\}$, where $\min\emptyset:=\infty$.
For two subgraphs $H,H'\subset G$, we define their graph distance by $d_G(H,H'):=\min\{d_G(u,v):\ u\in V_H,\ v\in V_{H'}\}$. With a slight abuse of notation, we sometimes write $v\in G$ to mean $v\in V$, and similarly set $|G|:=|V|$ for the cardinality of the vertex set. We denote by $\mathcal{G}$ the set of all finite graphs.

\subsection{Local Convergence of Rooted Graphs}\label{subsec:local-conv-rooted-graphs}
A rooted graph is a triple $G=(V,E,o)$ consisting of a (simple, locally finite) graph $(V,E)$ with a distinguished vertex $o\in V$, called the root. We write $G=(G,o)$ when we wish to emphasize the distinguished root $o$, and we denote by $[[G]]$ the corresponding unrooted graph. Given a rooted graph $G=(V,E,o)$ and $k\in\mathbb N_0$, we write $B_k(G)$ for the induced rooted subgraph consisting of all vertices $v\in V$ whose graph distance $d_G(o,v)$ from the root $o$ is at most $k$. Two rooted graphs $G=(V,E,o)$ and $G'=(V',E',o')$ are said to be isomorphic if there exists a bijection $\varphi:V\to V'$ with $\varphi(o)=o'$ such that
$$
\{u,v\}\in E \;\Longleftrightarrow\; \{\varphi(u),\varphi(v)\}\in E',\quad\text{for all }u,v\in V.
$$
In this case we write $G\cong G'$ and say that $\varphi$ is an isomorphism from $G$ to $G'$. We denote by $I(G,G')$ the set of all isomorphisms from $G$ to $G'$.
\begin{definition}\label{def:local-convergence}
Let $\mathcal G_\ast$ denote the set of isomorphism classes of connected, countable, rooted graphs. A sequence $\{G_n\}_{n\in\mathbb N}\subset\mathcal G_\ast$ is said to converge locally to $G\in\mathcal G_\ast$ if for every $k\in\mathbb N_0$ there exists $n_k\in\mathbb N$ such that $B_k(G_n)\cong B_k(G)$ for all $n\ge n_k$.
\end{definition}
\begin{remark}
The set $\mathcal G_\ast$ can be endowed with a natural metric. For two rooted graphs $G,G'\in \mathcal G_\ast$, define 
\be\label{eq:d_ast-unmarked}
d_\ast(G,G'):=\inf_{k\in\mathbb N_0}\left\{2^{-k}:B_k(G)\cong B_k(G')\right\}=\sum_{k=1}^\infty 2^{-k} \mathds{1}_{\{I(B_k(G),B_k(G'))=\emptyset\}}.
\ee
This turns $(\mathcal G_\ast, d_\ast)$ into a complete and separable metric space (see \cite{aldous2004objective}, Section~2.2). If one works with graphs of uniformly bounded degree, $(\mathcal G_\ast, d_\ast)$ is compact (see \cite{lovasz2012large}, Chapter~18.3.1).
\end{remark}
We now add a further layer of structure by considering marked graphs, in which each vertex carries an associated mark. In the context of a game, vertices represent players, and the marks encode both player-specific noise and individual actions.

Let $(\mathcal X,d)$ be a metric space. An $\mathcal X$-marked rooted graph is a pair $(G,x)=(G,o,x)$, where $G = (V,E,o)\in \mathcal G_\ast$ is a connected, countable, rooted graph and $x = (x_v)_{v\in V} \in \mathcal X^V$ is a family of vertex marks taking values in $\mathcal X$. For $(G,x)$ and $k\in\mathbb N_0$, we denote by $B_k(G,x)$ the induced $\mathcal X$-marked rooted subgraph obtained by restricting $G$ to the ball $B_k(G)$ around the root and equipping each vertex $v\in B_k(G)$ with its mark $x_v$.
Two $\mathcal X$-marked rooted graphs $(G,x)$ and $(G',x')$ are called isomorphic if there exists an isomorphism $\varphi : G \to G'$ such that $x_v = x'_{\varphi(v)}$ for all $v\in V(G)$. In this case we write $(G,x)\cong (G',x')$.

\begin{definition}\label{def:local-convergence-marked}
Let $(\mathcal X,d)$ be a metric space. Denote by $\mathcal G_\ast[\mathcal X]$ the set of isomorphism classes of $\mathcal X$-marked rooted graphs. A sequence $\{(G_n,x^n)\}_{n\in\mathbb N} \subset \mathcal G_\ast[\mathcal X]$ is said to converge locally to $(G,x)\in \mathcal G_\ast[\mathcal X]$ if for every $k\in\mathbb N_0$ and every $\varepsilon>0$ there exists $n_{k,\eps}\in\mathbb N$ such that for all $n\ge n_{k,\eps}$ there is an isomorphism $\varphi : B_k(G_n) \to B_k(G)$ with
$$
\max_{v\in B_k(G_n)} d\big(x^n_v,\, x_{\varphi(v)}\big) < \varepsilon.
$$
\end{definition}
\begin{remark}
As in the unmarked case, the space $\mathcal G_\ast[\mathcal X]$ can be equipped with a metric $d_\ast$ corresponding to this notion of convergence. Following \cite{lacker2023label}, one convenient choice extending \eqref{eq:d_ast-unmarked} is
\begin{equation}\label{eq:d-ast-marked}
d_\ast\big((G,x),(G',x')\big):= \sum_{k=1}^\infty 2^{-k} \Big( 1 \wedge \inf_{\varphi\in I(B_k(G),B_k(G'))}\max_{v\in B_k(G)} d\big(x_v,\, x'_{\varphi(v)}\big) \Big),
\end{equation}
where $\inf\emptyset:=\infty$. If $(\mathcal X,d)$ is complete and separable, then $(\mathcal G_\ast[\mathcal X],d_\ast)$ is also complete and separable (see \cite{bordenave2016lecture}, Lemma~3.4).
\end{remark}
\begin{remark}\label{rem:vertex-edge-marked}
We will also consider graphs that not only have vertex marks in a metric space $\mathcal{X}$, but also edge marks in another metric space $\mathcal{Y}$. Namely, an $(\mathcal X,\mathcal{Y})$-marked rooted graph is a triple $(G,x,y)$, where $G = (V,E,o)\in \mathcal G_\ast$ is a connected, countable, rooted graph, $x = (x_v)_{v\in V} \in \mathcal X^V$ is a family of vertex marks, and $(y_e)_{e\in E}\in \mathcal{Y}^E$ is a family of edge marks. We follow \cite{angel2018hyperbolic,bordenave2016lecture,lee2021conformal} in assigning one $\mathcal Y$-valued mark to each undirected edge, rather than two as, for instance, in \cite{Aldous2007,benjamini2015unimodular,van2024random}. The corresponding notion of isomorphism and the associated local metric are defined analogously to the $\mathcal X$-marked case, and we denote by $\mathcal{G}_\ast [\mathcal{X},\mathcal Y]$ the set of isomorphism classes of $(\mathcal X,\mathcal{Y})$-marked rooted graphs.
\end{remark}


\subsection{Local Weak Convergence}\label{subsec:local-weak-conv}
Up to now, we have worked with connected graphs, but for local weak convergence one often starts from a possibly disconnected graph and then simply views it through the rooted connected component of a chosen vertex. We make this concept precise following \cite{lacker2023label,lovasz2012large,van2024random}.

Let $G=(V,E)$ be a graph and $v\in V$ a vertex. We write $C_v(G)$ for the (isomorphism class of the) connected component of $v$, that is, the rooted graph obtained by taking all vertices $u\in V$ that are connected to $v$ by a path, and viewing $v$ as the root. In particular, $C_v(G)\in \mathcal G_\ast$. For a finite graph $G$, let $U$ be a uniformly random vertex of $G$ and denote by $C_U(G)$ the corresponding $\mathcal G_\ast$-valued random variable.
In the marked setting, for an $\mathcal X$-marked rooted graph $(G,x)$, we analogously set $C_v(G,x) := (C_v(G),\, x_{C_v(G)})$ and $C_U(G,x) :=(C_U(G),\, x_{C_U(G)})$, where $x_{C_v(G)}$ and $x_{C_U(G)}$ denote the corresponding restriction of the marks.

All the graph sequences we considered so far were deterministic. Throughout the rest of this section, we work with random graph sequences defined on a probability space with probability measure $\Q$. Just as for random variables, one typically considers convergence in distribution, convergence in probability, and almost sure convergence. Denote by $C_b(\mathcal G_\ast)$ the set of continuous, bounded functions on $\mathcal G_\ast$. Recall that $\mathcal{G}$ denotes the set of all finite (simple, possibly disconnected) graphs.

\begin{definition}\label{def:local-weak-convergence}
Let $\{G_n\}_{n\in\mathbb N}\subset\mathcal G$ be a sequence of finite random graphs $G_n=(V_n,E_n)$ and let $G$ be a random rooted graph in $\mathcal G_\ast$. Then:
\begin{enumerate}
    \item We say that $\{G_n\}_{n\in\mathbb N}$ converges in distribution in the local weak sense to $G$ if
    \[
    \E_{\Q}\Big[\frac{1}{|V_n|}\sum_{v \in V_n} h\big(C_v(G_n)\big)\Big]\xrightarrow[n\to\infty]{}\E_{\Q}\big[h(G)\big],\quad \text{for all } h \in C_b(\mathcal G_\ast).
    \]
    If $U^n$ is a uniformly random vertex of $G_n$, then, equivalently, $\{C_{U^n}(G_n)\}_{n\in\N}$ converges in distribution to $G$ in $\mathcal{G}_\ast$.
    \item We say that $\{G_n\}_{n\in\mathbb N}$ converges in $\Q$-probability in the local weak sense to $G$ if
    \[
    \frac{1}{|V_n|}\sum_{v \in V_n} h\big(C_v(G_n)\big)\xrightarrow[n\to\infty]{\ \text{prob.}\ }\E_{\Q}\big[h(G)\big],\quad \text{for all } h \in C_b(\mathcal G_\ast).
    \]
    \item We say that $\{G_n\}_{n\in\mathbb N}$ converges $\Q$-almost surely in the local weak sense to $G$ if
    \[
    \frac{1}{|V_n|}\sum_{v \in V_n} h\big(C_v(G_n)\big)\xrightarrow[n\to\infty]{\text{a.s.} }\E_{\Q}\big[h(G)\big],\quad \text{for all } h \in C_b(\mathcal G_\ast).
    \]
\end{enumerate}
\end{definition}

\begin{remark}
The three notions of local weak convergence above form a hierarchy:
$\Q$-almost sure convergence implies convergence in $\Q$-probability, which in turn
implies convergence in distribution. For deterministic graph sequences, however, these distinctions disappear. If the sequence $\{G_n\}_{n\in\N}$ is deterministic, all randomness lies in the limit object $G$, and all three notions of convergence reduce to the single requirement
\[
\frac{1}{|V_n|}\sum_{v \in V_n} h\big(C_v(G_n)\big)\xrightarrow[n\to\infty]{}\E_{\Q}\big[h(G)\big],\quad \text{for all } h \in C_b(\mathcal G_\ast).
\]
\end{remark}

\begin{remark}\label{rem:local-weak-convergence-marked}
The concept of local weak convergence can be extended to the marked setting. Let $(\mathcal X,d)$ be a metric space. For $n\in\N$, let $(G_n,x^n)$ be a finite random graph $G_n = (V_n,E_n)$ equipped with random vertex marks $x^n = (x_v^n)_{v\in V_n}\in \mathcal X^{V_n}$. Let $(G,x)$ be a random rooted graph $G$ together with random marks $x=(x_v)_{v\in G}$, that is, a $\mathcal G_\ast[\mathcal X]$-valued random variable. Then Definition~\ref{def:local-weak-convergence} carries over simply by replacing $C_v(G_n)$ with the marked component $C_v(G_n,x^n)$ and $h\in C_b(\mathcal G_\ast)$ with $h\in C_b(\mathcal G_\ast[\mathcal X])$. For example, convergence in distribution in the local weak sense of the marked graphs $\{(G_n,x^n)\}_{n\in\mathbb N}$ to $(G,x)$ means that
\[
\E_{\Q}\Big[\frac{1}{|V_n|}\sum_{v \in V_n} h\big(C_v(G_n,x^n)\big)\Big]\xrightarrow[n\to\infty]{}\E_{\Q}\big[h(G,x)\big],\quad \text{for all } h \in C_b\big(\mathcal G_\ast[\mathcal X]\big).
\]
\end{remark}
The concepts presented next in Sections~\ref{subsec:limits} and \ref{subsec:hyperfiniteness} are only needed for Corollary~\ref{cor:convergence}, where we construct a locally weakly convergent sequence of finite games by sampling from a fixed random rooted graph. They are not used elsewhere in the paper, and may be skipped without affecting the readability of the other results.

\subsection{Limits of Sparse Graph Sequences}\label{subsec:limits}
The limits of locally weakly convergent graph sequences are random rooted graphs in $\mathcal G_\ast$, or, equivalently, probability measures on $\mathcal G_\ast$. So far, however, we have not put any measure-theoretic structure on $\mathcal G_\ast$ itself. Recall that $d_\ast$ makes $(\mathcal G_\ast,d_\ast)$ a complete and separable metric space, and denote by $\mathcal{B}(\mathcal{G}_\ast)$ the associated Borel $\sigma$-algebra on $\mathcal G_\ast$. An important observation is that not every probability measure on $(\mathcal G_\ast,\mathcal{B}(\mathcal{G}_\ast))$ arises as a local weak limit. More precisely, since the root in Definition~\ref{def:local-weak-convergence} is chosen uniformly at random, any limiting law must satisfy a corresponding invariance property. This requirement is captured by the concept of unimodularity, which we introduce following \cite{Aldous2007,bordenave2016lecture}. 

A doubly-rooted graph $G=(V,E,o_1,o_2)$ consists of a (simple, locally finite) graph $(V,E)$ and two roots $o_1,o_2\in V$. We write $G=(G,o_1,o_2)$ when we wish to emphasize the distinguished roots $o_1$, $o_2$. Two doubly-rooted graphs $G=(V,E,o_1,o_2)$ and $G'=(V',E',o_1',o_2')$ are said to be isomorphic if there exists a bijection $\varphi:V\to V'$ with $\varphi(o_1)=o_1'$ and $\varphi(o_2)=o_2'$ such that
$$
\{u,v\}\in E \;\Longleftrightarrow\; \{\varphi(u),\varphi(v)\}\in E',\quad\text{for all }u,v\in V.
$$
Let $\mathcal G_{\ast\ast}$ denote the set of isomorphism classes of connected, countable, doubly-rooted graphs. As for rooted graphs, $\mathcal G_{\ast\ast}$ can be equipped with a local metric $d_{\ast\ast}$ that gives rise to a Borel $\sigma$-algebra on $\mathcal G_{\ast\ast}$.
\begin{definition}\label{def:MTP}
A probability measure $\sigma$ on $(\mathcal G_\ast,\mathcal{B}(\mathcal{G}_\ast))$ is called unimodular if it obeys the following mass-transport principle: for all Borel-measurable $F: \mathcal G_{\ast\ast} \to [0,\infty]$, 
\be\label{eq:MTP}
\int_{\mathcal G_{\ast}} \sum_{v\in G} F(G,o,v)\sigma(d(G,o))= \int_{\mathcal G_{\ast}} \sum_{v\in G} F(G,v,o)\sigma(d(G,o)).
\ee
If $G\in\mathcal G_\ast$ is a random rooted graph with law $\sigma=\mathcal{L}(G)$ and $\sigma$ is a unimodular distribution, we say that $G$ is a unimodular random graph.
\end{definition}
As observed by \citet{benjamini2001recurrence}, any random rooted graph that is the local weak limit of a finite graph sequence is unimodular, as stated in the following lemma.
\begin{lemma} 
Let $\{G_n\}_{n\in\N}\subset\mathcal G$ be a sequence of finite random graphs, and let $G\in\mathcal{G}_\ast$ be a random rooted graph. If $\{G_n\}_{n\in\mathbb N}$ converges in distribution in the local weak sense to $G$, then $G$ is unimodular.
\end{lemma}

\subsection{Hyperfiniteness and Sampling}\label{subsec:hyperfiniteness}
While unimodularity is a necessary condition for a random rooted graph $G\in\mathcal G_\ast$ to arise as the local weak limit of a sequence of finite graphs, it is not sufficient in general \cite{bowen2024aldous,bowen2024aldous2}. A convenient additional assumption ensuring sufficiency and thereby guaranteeing approximability of $G$ by finite graphs is hyperfiniteness \cite{Aldous2007,schramm2011hyperfinite}. We follow the presentation of \cite{angel2018hyperbolic,lee2021conformal}, which aligns best with our framework. Recall the definition of vertex-edge-marked graphs from Remark~\ref{rem:vertex-edge-marked}.

\begin{definition}\label{def:percolation}
Let $G=(V,E,o)\in\mathcal G_\ast$ be a unimodular random graph. A percolation $\xi$ on $G$ is a $\{0,1\}$-valued vertex-edge marking $\xi=(x,y)$ such that the random rooted vertex-edge-marked graph $(G,x,y)$ with $x = (x_v)_{v\in V} \in  \{0,1\}^V$ and $(y_e)_{e\in E}\in \{0,1\}^E$ is unimodular. That is, the law of $(G,x,y)$ satisfies Definition~\ref{def:MTP} with $\mathcal{G}_\ast$, $\mathcal{G}_{\ast\ast}$, and $F: \mathcal G_{\ast\ast} \to [0,\infty]$ replaced by $\mathcal{G}_\ast[\mathcal{X},\mathcal{Y}]$, $\mathcal{G}_{\ast\ast}[\mathcal{X},\mathcal{Y}]$, and $F:\mathcal G_{\ast\ast}[\mathcal{X},\mathcal{Y}] \to [0,\infty]$ for $\mathcal{X}=\mathcal{Y}=\{0,1\}$, respectively. Here, we assume consistency of vertex and edge marks in the sense that an edge can be marked $1$ only if both its endpoints are marked $1$, i.e.,~$y_{\{u,v\}}=1$ implies $x_u=x_v=1$ for every $\{u,v\}\in E$.
\end{definition}
It may be more intuitive to view a percolation $\xi$ as specifying a random subgraph $H_\xi\subset G$ where precisely those vertices and edges with mark~$1$ are kept. We say that $\xi$ is finitary if all connected components of $H_\xi$ are finite $\Q$-almost surely.
\begin{definition}\label{def:hyperfinite}
A unimodular random graph $G\in\mathcal G_\ast$ is called hyperfinite if there exists a sequence $(\xi_n)_{n\in\N}$ of finitary percolations on $G$ such that the associated random subgraphs form an increasing sequence $H_{\xi_1}\subset H_{\xi_2}\subset\ldots$ with $\bigcup_{n\in\N}H_{\xi_n}=G$, $\Q$-almost surely. In this case, $\{H_{\xi_n}\}_{n\in\N}$ is called a finitary exhaustion.
\end{definition}

Percolations and hyperfiniteness can be analogously defined for vertex-marked graphs. The main point for us is that hyperfiniteness yields an approximation by finite unimodular
random graphs $\{H_n\}_{n\in\N}$ (see \cite{lee2021conformal}, Corollary~4.2).

\begin{lemma}\label{lemma:sample-rooted}
Let $G=(G,o)\in\mathcal G_\ast$ be a hyperfinite unimodular random graph with finitary exhaustion $\{H_{\xi_n}\}_{n\in\N}$. Then the sequence of finite unimodular random graphs $\{H_n\}_{n\in\N}\subset\mathcal G_\ast$ given by $H_n:=C_o(H_{\xi_n})$ converges $\Q$-almost surely and hence also in distribution to $G$ in $(\mathcal{G}_\ast,d_\ast)$.
\end{lemma}
Lemma~\ref{lemma:sample-rooted} yields a sequence of finite unimodular random graphs $\{H_n\}_{n\in\N}$ that converges in distribution in $\mathcal{G}_\ast$ to a given hyperfinite unimodular random graph $G$. Since $\{H_{\xi_n}\}_{n\in\N}$ is finitary, the mass-transport principle implies that the root is uniformly distributed on its connected component $H_n=C_o(H_{\xi_n})$ (see \cite{angel2018hyperbolic}, Lemma~3.1). Therefore, by Definition~\ref{def:local-weak-convergence}, forgetting the root yields an unrooted sequence of finite random graphs converging to $G$ in distribution in the local weak sense.

\begin{corollary}\label{cor:sample-unrooted}
Let $G=(G,o)\in\mathcal{G}_\ast$ be a hyperfinite unimodular random graph with finitary exhaustion $(H_{\xi_n})_{n\in\N}$. Then the sequence of finite random unrooted graphs $\{ G_n\}_{n\in\N}\subset\mathcal{G}$ given by $G_n:=[[C_o(H_{\xi_n})]]$ converges
in distribution in the local weak sense to $G$.
\end{corollary}

\begin{remark}\label{rem:sample-marked}
Lemma~\ref{lemma:sample-rooted} and Corollary~\ref{cor:sample-unrooted} can be extended to the marked setting. Let $(\mathcal X,d)$ be a metric space. If $(G,x)=(G,o,x)\in\mathcal G_\ast[\mathcal{X}]$ is a hyperfinite unimodular random $\mathcal{X}$-marked graph with finitary exhaustion $\{H_{\xi_n}\}_{n\in\N}$, then the sequence of finite unimodular random $\mathcal{X}$-marked graphs $\{(H_n,x_{H_n})\}_{n\in\N}$ given by $H_n:=C_o(H_{\xi_n})$ converges in distribution to $(G,x)$ in $\mathcal G_\ast[\mathcal{X}]$. Similarly, by Lemma~3.1 of \cite{angel2018hyperbolic}, this implies that the sequence of finite random unrooted $\mathcal{X}$-marked graphs $\{(G_n,x_{G_n})\}_{n\in\N}$ given by $G_n:=[[C_o(H_{\xi_n})]]$ converges in distribution in the local weak sense to $(G,x)$.
\end{remark}

\section{Model Setup}\label{sec:model}
In this section we present a general framework for stochastic games on large sparse graphs, which includes both continuous-time and discrete-time games and allows heterogeneous, path-dependent utility functionals. We also provide several examples.

\subsection{Formulation of the Game}
Let $T>0$ denote a finite deterministic time horizon and fix a filtered probability space $(\Omega, \mathcal F, \F:=(\mathcal{F}_t)_{0\leq t\leq T}, \P)$. We consider a dynamic stochastic game on a fixed (locally finite, simple) graph $G=(V,E)$, where each player interacts with their finitely many neighbors. Here we allow both finite and countably infinite graphs and thus define the finite-player and infinite-player game at the same time. We label the  players by $v\in V$ (or $v\in G$). Define the set of admissible actions
\be\label{eq:A}  
\mathcal A:=\left\{\tilde{a}:\Omega\times [0,T]\to\R \, \Big| \, \tilde{a} \textrm{ is } \mathbb{F}\textrm{-prog.~measurable, } \,\|\tilde a\|^2_{\mathcal{A}}:=\E_{\P}\Big[\int_0^T  \tilde{a}(t)^2 dt  \Big]  <\infty \right\},
\ee
where throughout, variables with a tilde denote generic elements in $\mathcal{A}$. The set $\mathcal{A}$ can be turned into a Hilbert space by equipping it with the inner product
\be
\big\langle  \tilde a_1,  \tilde a_2\big\rangle_{\mathcal{A}}:=\E_{\P}\left[\int_0^T \tilde a_1(t) \tilde a_2(t)dt\right],\quad \tilde a_1, \tilde a_2\in\mathcal{A}.
\ee
\begin{remark}\label{rem:discrete-time}
Our framework also includes discrete-time stochastic games. Namely, for a fixed finite partition of $[0,T]$ given by $\mathbb T:=\{0=t_0<t_1<\ldots<t_N=T\}$ with $N\in\N_0$, one can analogously define the set of admissible actions by 
\be\label{eq:A-discrete-time}  
\mathcal A:=\bigg\{\tilde{a}:\Omega\times \mathbb{T}\to\R \, \Big| \, \tilde a(t_j)\ \text{is }\mathcal F_{t_j}\text{-measurable for each }j, \,\,\|\tilde a\|^2_{\mathcal{A}}:=\E_{\P}\Big[\sum_{j=0}^N  \tilde{a}(t_j)^2  \Big]  <\infty \bigg\},
\ee
with corresponding inner product rendering it a Hilbert space given by 
\be
\big\langle  \tilde a_1,  \tilde a_2\big\rangle_{\mathcal{A}}:=\E_{\P}\Big[\sum_{j=0}^N \tilde a_1(t_j) \tilde a_2(t_j)\Big],\quad \tilde a_1, \tilde a_2\in\mathcal{A}.
\ee
In particular, the case where $T=0$ and $N=0$ then yields a static game. Throughout the paper, we have both the continuous-time and the discrete-time formulation in mind, and all results and proofs hold in either setting.
\end{remark}
\begin{definition}\label{def:action-profile} 
An action profile is a family $a=(a_v)_{v\in G}$ of actions $a_v\in\mathcal{A}$. Define the set of feasible action profiles as
\be\label{eq:AG}
\mathcal{A}^G := \Big\{ a=(a_v)_{v\in G} \,\Big|\, a_v\in\mathcal A,\ \|a\|_{\mathcal{A}^G}:=\sup_{v\in G}\|a_v\|_{\mathcal A}<\infty\Big\}.
\ee
rendering $\mathcal{A}^G$ a Banach space.
\end{definition}
Each player $v\in G$ selects an action $a_v$ from their individual set of admissible actions $\mathcal{A}_v\subset\mathcal{A}$, which may be heterogeneous across the players. 
Define the corresponding set of admissible action profiles as
\be\label{eq:AGad}
\mathcal{A}_{ad}^G:=\left\{a\in\mathcal{A}^G\Big| a_v\in\mathcal{A}_v\textrm{ for all }v\in G\right\}.
\ee
In the sparse setting, each player interacts with at most finitely many other players. The interaction effects experienced by player $v\in G$ are defined as the local aggregate
\begin{equation}\label{eq:local-aggregate}
z_v(a):= (\boldsymbol{G}a)(v) := \frac{\mathds{1}_{\{\deg_G(v)>0\}}}{\deg_G(v)}\sum_{u\sim v} a_u,\quad a\in\mathcal{A}^G,\ v\in G,
\end{equation}
where we use the convention that $0/0:=0$ and note that the denominator in \eqref{eq:local-aggregate} is finite, since $G$ is locally finite by assumption.
\begin{remark}\label{rem:local-aggregate-operator}
Note that the local aggregate operator $\boldsymbol{G}$ on $\mathcal{A}^G$ from \eqref{eq:local-aggregate} is linear. Its operator norm given by
\be\label{eq:operator-norm}
\|\boldsymbol{G}\|_{\text{op}}:=\sup\left\{ \|\boldsymbol{G} a\|_{\mathcal{A}^G}:a\in \mathcal{A}^G, \|a\|_{\mathcal{A}^G}\leq 1 \right\}
\ee
is equal to 1 (unless $G$ has no edges, then $\|\boldsymbol{G}\|_{\text{op}}= 0$), thus, $\boldsymbol{G}$ is bounded. We could have also defined $\boldsymbol{G}$ in another way. For instance, if the vertex degrees of $G$ are uniformly bounded by $D\in\N$, the alternative definition 
$$
(\boldsymbol{G}a)(v) := \frac{1}{D}\sum_{u\sim v} a_u,\quad a\in\mathcal{A}^G,\ v\in G,
$$
yields a linear bounded operator with $\|\boldsymbol{G}\|_{\text{op}}\leq 1$, and all the subsequent analysis carries over with minimal modifications.
\end{remark}
Next, let $\theta=(\theta_v)_{v\in G}\in\mathcal{A}^G$ be a heterogeneity profile consisting of processes $\theta_v\in\mathcal{A}$. These can contain both idiosyncratic and common noise and thereby allow us to incorporate heterogeneity among the players. Consider a universal utility functional 
\be \label{eq:U}
U:\mathcal{A}\times\mathcal{A}\times\mathcal{A}\to\R,
\ee
acting directly on player actions, local aggregates, and heterogeneity processes. Let the individual utility functional of player $v$ on $\mathcal{A}_v$ be given by 
\be 
a_v\mapsto U\left(a_v,z_v(a),\theta_v\right).
\ee
In line with \cite{neuman2025stochastic,parise2023graphon}, we denote this game by $\mathbb{G}(\mathcal{A}_{ad}^G,U,\theta,G)$.
\begin{definition} \label{def:Nash}
An admissible action profile $(\bar{a}_v)_{v\in G}\in\mathcal{A}_{ad}^G$ is called a Nash equilibrium of the game $\mathbb{G}(\mathcal{A}_{ad}^G,U,\theta,G)$ if for every $v\in G$ the action $\bar a_v$ satisfies 
\be \label{eq:Nash}
\bar a_v=\argmax_{\tilde {a}\in\mathcal{A}_v} \ U\left(\tilde{a}, z_v(\bar a),\theta_v\right).
\ee
\end{definition}
We will also need the more general notion of approximate Nash equilibria.
\begin{definition}\label{def:eps-Nash} Let $\eps>0$.
An admissible action profile $(a_v)_{v\in G}\in\mathcal{A}_{ad}^G$ is called an $\eps$-Nash equilibrium of the game $\mathbb{G}(\mathcal{A}_{ad}^G,U,\theta,G)$ if for every $v\in G$ the action $ a_v$ satisfies 
\begin{equation}\label{eq:eps-Nash}
U(a_v,z_v(a),\theta_v)\ge
\sup_{\tilde a\in\mathcal A_v} U\left(\tilde a,z_v(a),\theta_v\right) -\varepsilon.
\end{equation}
\end{definition}

\subsection{Examples}
Note that our formulation does not model state variables explicitly. Instead, each player's payoff is specified through the universal functional $U$ in \eqref{eq:U}, which acts directly on progressively measurable processes. This yields a flexible setup that encompasses, for instance, the following classes of games.

\begin{example}\label{example:U-1}
A basic class of games is given by functionals of the form
\[
U(\tilde a,\tilde z,\tilde\theta)
= \E_\P\!\left[\int_0^T f_U\big(\tilde a(t),\tilde z(t),\tilde\theta(t)\big)\,dt\right],
\]
where $f_U:\R^3\to\R$ specifies a running payoff. By considering a discrete time grid $\mathbb T$ equipped with the counting measure in the sense of Remark~\ref{rem:discrete-time}, one recovers discrete-time games.
\end{example}

\begin{example}\label{example:U-2}
A further important family with various applications is given by linear--quadratic Volterra functionals,
\[
U(\tilde a,\tilde z,\tilde\theta)
=
\langle \tilde a, \boldsymbol A_1 \tilde a\rangle_{\mathcal A}
+\langle \tilde a, \boldsymbol A_2 \tilde z\rangle_{\mathcal A}
+\langle \tilde z, \boldsymbol A_2 \tilde a\rangle_{\mathcal A}
+\langle \tilde z, \boldsymbol A_3 \tilde z\rangle_{\mathcal A}
+\langle \tilde a,\tilde\theta\rangle_{\mathcal A}
+\langle \tilde z,\tilde\theta^\ast\rangle_{\mathcal A},
\]
where $\boldsymbol A_1,\boldsymbol A_2,\boldsymbol A_3$ are Volterra operators on $L^2([0,T])$ extended to $\mathcal{A}$ in the natural way and $\tilde\theta^\ast\in\mathcal A$ is an exogenous input (e.g.~common noise). The Volterra structure naturally encodes memory, allowing $U$ to depend on the past of $\tilde a$ and $\tilde z$ in a non-Markovian way. In the finite-player setting, explicit Nash equilibria for this class were derived in \cite{neuman2024stochastic}.
\end{example}

\begin{example}\label{example:U-3}
A large class of games with linear state variables can be rewritten in the above state-free form by explicitly solving the state equation and substituting the solution into the objective. Let $G$ be any simple, locally finite graph and fix square-integrable Volterra kernels $K,L:[0,T]^2\to\R$ together with idiosyncratic noise processes $\eta=(\eta_v)_{v\in G}\in\mathcal A^G$. 
Given an action profile $a=(a_v)_{v\in G}\in\mathcal A^G$, suppose each player $v\in G$ has a controlled state process $X_v\in\mathcal{A}$ following the linear dynamics
\begin{equation}\label{eq:linear-states}
X_v(t)=\int_0^t K(t,s)X_v(s)ds+\int_0^t L(t,s)a_v(s)ds+\eta_v(t),\quad t\in[0,T].
\end{equation}
Define the corresponding local aggregate state $Z_v\in\mathcal A$ by
\[
Z_v := z_v(X) = (\boldsymbol G X)(v),\quad\text{i.e.,}\quad Z_v(t)=\frac{\mathds{1}_{\{\deg_{G}(v)>0\}}}{\deg_{G}(v)}\sum_{u\sim v} X_u(t),\quad t\in[0,T].
\]
Let $R$ denote the resolvent of the kernel $(-K)$ and set
\[
(R\star L)(t,s):=\int_0^t R(t,u)L(u,s)du,\quad t,s\in [0,T].
\]
Then \eqref{eq:linear-states} admits the explicit representation
\begin{equation}\label{eq:linear-states-resolved}
X_v(t)=\int_0^t \big(L-(R\star L)\big)(t,s)a_v(s)ds-\int_0^t R(t,s)\eta_v(s)ds+\eta_v(t),\quad t\in[0,T],
\end{equation}
by Chapter~9.3 of \cite{gripenberg1990volterra}. Consider payoffs of the classical form
\begin{equation}\label{eq:Jv}
J_v(a_v;(a_u)_{u\neq v})=\E_\P\left[\int_0^T f_J\big(t,X_v(t),Z_v(t),a_v(t)\big)dt+g_J\big(X_v(T),Z_v(T)\big)\right],
\end{equation}
with running payoff $f_J:\R^4\to\R$ and terminal payoff $g_J:\R^2\to\R$. Solving for the state variables via \eqref{eq:linear-states-resolved} and substituting into \eqref{eq:Jv} yields a representation of the form
\[
J_v(a_v;(a_u)_{u\neq v}) = U\big(a_v, z_v(a), \theta_v\big),
\]
for a suitable universal functional $U$ and a suitable heterogeneity profile $(\theta_v)_{v\in G}\in\mathcal A^G$. Hence, games with linear state dynamics can be embedded into our general framework once the associated state equations are solved explicitly.
\end{example}

\section{Main Results}\label{sec:results}
This section presents our main results, including the existence and uniqueness of equilibria, correlation decay of equilibrium actions, local approximation and local reconstruction of equilibria, and convergence of equilibria along locally weakly convergent graph sequences.

\subsection{Existence and Uniqueness}\label{subsec:existence-uniqueness}
In line with \cite{neuman2025stochastic,parise2023graphon}, we focus on continuously differentiable and strongly concave utility functionals to study the Nash equilibrium of the game and its properties.

\begin{assumption}\label{assum:U}
For all $\tilde z,\tilde\theta\in\mathcal{A}$, the utility functional $U(\tilde a,\tilde z,\tilde \theta)$ in \eqref{eq:U} is continuously G\^ateaux differentiable in $\tilde a$ and strongly concave in $\tilde a$ with strong concavity constant $\gamma_U>0$, that is, $\smash{U(\tilde a,\tilde z,\tilde \theta)+\frac{\gamma_U}{2}\|\tilde a\|^2_{\mathcal{A}}}$ is concave in $\tilde a$. Moreover, $\smash{\nabla_{\tilde a} U(\cdot,\tilde z,\tilde \theta)}$ is Lipschitz continuous in $\tilde z,\tilde\theta$ with constants $\ell_U,\ell_\theta$, and it holds that 
\be\label{eq:rho}
\rho_{\raisebox{0ex}{\tiny\eqref{eq:rho}}}:=\ell_U/\gamma_U<1.
\ee
\end{assumption}

\begin{assumption}\label{assum:argmax}
For each player $v\in G$, the set of admissible actions $\mathcal{A}_v$ from \eqref{eq:AGad} is nonempty, convex, and closed. Moreover, there exist $\tilde z^0, \tilde\theta^0 \in\mathcal{A}$ such that 
$$
\big(\argmax_{\tilde{a}\in \mathcal{A}_v} U(\tilde{a}, \tilde z^0,\tilde\theta^0)\big)_{v\in G}\in\mathcal{A}^G.
$$
\end{assumption}
We note that the second part of Assumption~\ref{assum:argmax} always holds in the finite-player case. The following theorem establishes existence and uniqueness of a Nash equilibrium as introduced in Definition~\ref{def:Nash}.
\begin{theorem}\label{thm:NE}
Suppose that Assumptions~\ref{assum:U} and \ref{assum:argmax} hold. Then the game $\mathbb{G}(\mathcal{A}_{ad}^G,U,\theta,G)$ admits a unique Nash equilibrium for any simple, locally finite graph $G$.
\end{theorem}
The proof of Theorem~\ref{thm:NE} is given in Section~\ref{sec:proofs-existence-uniqueness}. 
\begin{remark}
The condition $\rho_{\raisebox{0ex}{\tiny\eqref{eq:rho}}}<1$ will be essential throughout the paper. It ensures that the optimization problem can be cast via a contraction principle: the Nash equilibrium is characterized as the unique fixed point of an appropriate best-response operator, with $\rho_{\raisebox{0ex}{\tiny\eqref{eq:rho}}}$ acting as a contraction constant. This contraction property not only yields existence and uniqueness via Banach's fixed-point theorem, but also provides geometric error bounds for the associated Picard iterates. Such bounds allow for efficient approximation of the Nash equilibrium and will serve as a central mechanism in the proofs of our other results (see Propositions~\ref{prop:correlation-decay}, \ref{prop:local-approximation}, \ref{prop:epsilon-Nash}, \ref{prop:local-reconstruction}, and Theorem~\ref{thm:convergence}).
\end{remark}

\subsection{Correlation Decay}\label{subsec:correlation-decay}
The following result shows that, under independence of the players' heterogeneity processes, the equilibrium actions exhibit exponential decay of correlations with graph distance.

\begin{assumption}\label{assum:Av}
There exists a constant $M_{\raisebox{0ex}{\tiny\eqref{eq:A_M}}}>0$ such that for every $v\in G$, the set of admissible actions $\mathcal{A}_v$ introduced before \eqref{eq:AGad} satisfies
\be\label{eq:A_M}
\mathcal{A}_v \subset \mathcal{A}_{M_{\raisebox{0ex}{\tiny\eqref{eq:A_M}}}}:=\big\{ \tilde a \in \mathcal A \,\big|\,  \|\tilde a\|_\mathcal{A}   \leq M_{\raisebox{0ex}{\tiny\eqref{eq:A_M}}} \big\}.
\ee
\end{assumption}
For a graph $G$ and a subgraph $H\subset G$, denote by $a_H\in\mathcal{A}^H$ the restriction of an action profile $a\in\mathcal{A}^G$ to $H$. Consistent with \eqref{eq:AG}, endow $\mathcal A^{H}$ with the supremum norm $\|a\|_{\mathcal{A}^H}:=\sup_{v\in H}\|a_v\|_{\mathcal{A}}$. Similarly, equip the associated path space $\mathcal{L}^H:=(L^2([0,T]))^H$ with the norm $\|\psi\|_{\mathcal{L}^H}:=\sup_{v\in H}\|\psi_v\|_{L^2}$ for $\psi\in\mathcal{L}^H$ and let $d_H$ be the induced metric
\be\label{eq:dH}
d_H(\psi,\psi'):=\sup_{v\in H}\|\psi_v-\psi'_v\|_{L^2}.
\ee
For a function $f:\mathcal{L}^H\to\R$ that is bounded and Lipschitz with respect to $d_{H}$, define the bounded Lipschitz norm 
\be \label{eq:BL-norm}
\|f\|_{\text{BL}}:=\max\big\{\|f\|_\infty,\mathrm{Lip}(f)\big\},
\ee
where $\|f\|_\infty$ is the supremum norm and $\mathrm{Lip}(f)$ is the Lipschitz constant of $f$. Recall the definition of the graph distance $d_G$ from Section~\ref{subsec:graphs}.
\begin{proposition}\label{prop:correlation-decay}
Suppose that Assumptions~\ref{assum:U}, \ref{assum:argmax}, and \ref{assum:Av} hold. Assume that the heterogeneity processes $(\theta_v)_{v\in G}$ are independent under $\P$ and that $\smash{\argmax_{\tilde{a}\in \mathcal{A}_v} U(\tilde{a}, \tilde z,\theta_v)}$ is $\sigma(\tilde z,\theta_v)$-measurable for all $\tilde z\in\mathcal{A}$ and $v\in G$. Let $\bar a\in\mathcal A^G$ denote the unique Nash equilibrium of $\mathbb{G}(\mathcal{A}_{ad}^G,U,\theta,G)$. For $i=1,2$, let $H_i\subset G$ be a finite induced subgraph and let $f_i:\mathcal L^{H_i}\to\R$ be bounded and Lipschitz with respect to $d_{H_i}$. Then, 
\[
\big|\mathrm{Cov}_{\P}\big(f_1(\bar a_{H_1}),f_2(\bar a_{H_2})\big)\big|
\le 2\rho_{\raisebox{0ex}{\tiny\eqref{eq:rho}}}^{k} M_{\raisebox{-0.3ex}{\tiny\eqref{eq:A_M}}}(|H_1|+|H_2|)\,\|f_1\|_{\text{BL}}\|f_2\|_{\text{BL}},
\]
where 
\be\label{eq:graph-distance}
k=\big\lceil  d_G(H_1,H_2)/2\big\rceil.
\ee
\end{proposition}
The proof of Proposition~\ref{prop:correlation-decay} is given in Section~\ref{sec:proofs-propositions}.

\begin{figure}[htb]
\centering
\includegraphics[width=0.95\linewidth]{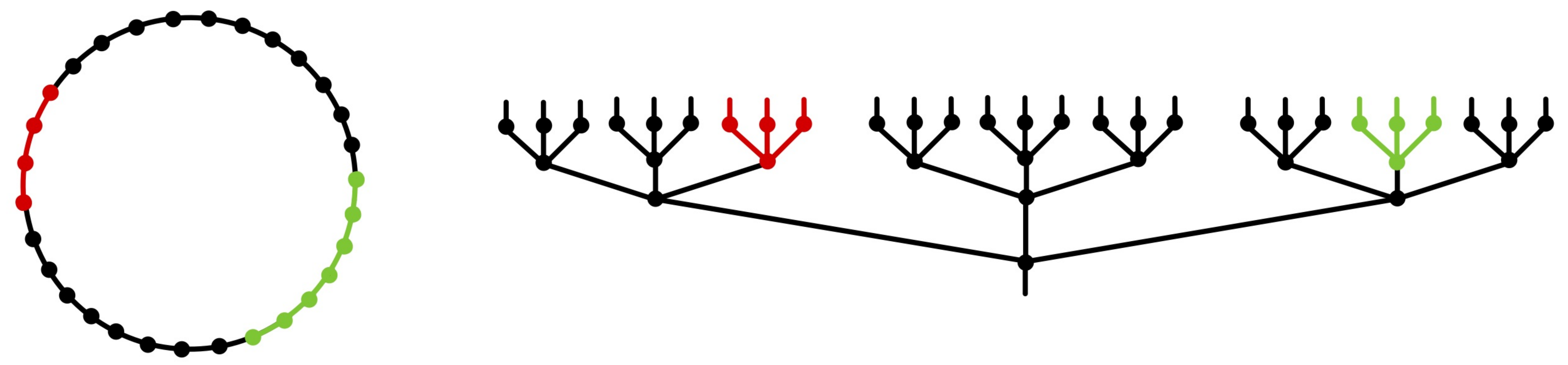}
\caption{Left: the circle graph with 30 vertices $G=C_{30}$ with two induced subgraphs $H_1\subset G$ (red) and $H_2\subset G$ (green). It holds that $d_G(H_1,H_2)=9$ and thus $k=5$ in \eqref{eq:graph-distance}. Right: an excerpt of the 4-regular tree $G'=T_4$ with two induced subgraphs $H'_1\subset G'$ (red) and $H'_2\subset G'$ (green). It holds that $d_{G'}(H'_1,H'_2)=4$ and thus $k=2$ in \eqref{eq:graph-distance}.}
\end{figure}
\begin{remark}
Proposition~\ref{prop:correlation-decay} establishes that correlations in the Nash equilibrium decay exponentially fast with respect to graph distance. Importantly, this is a statement for the equilibrium controls on the full time horizon. That is, the covariance is taken between functionals of entire trajectories of $(\bar a_u)_{u\in H_i}$, rather than describing a decay of correlations in time. This is conceptually different from correlation decay results for interacting particle systems (see \cite{lacker2023local}, Lemmas~5.1 and 5.2).
\end{remark}

\subsection{Local Approximation}\label{subsec:local-approximation}
It can be desirable to approximate the equilibrium action of a player $v\in G$ using only local information in a neighborhood of $v$. The next result formalizes such a local approximation of the Nash equilibrium. 

For $v\in G$ and $k\in\N_0$, denote the $k$-ball around $v$ by $B_k(G,v)$. Given a local action profile $a=(a_u)_{u\in B_k(G,v)}\in \mathcal A^{B_k(G,v)}$, define the truncated local aggregate by
\be\label{eq:local-aggregate-trunc}
z^{(v,k)}_u(a)
:= (\boldsymbol G^{(v,k)}a)(u)
:=\frac{\mathds{1}_{\{\deg_G(u)>0\}}}{\deg_G(u)}\sum_{\substack{w\sim u\\ w\in B_k(G,v)}} a_w,
\quad u\in B_k(G,v).
\ee
Note that we normalize by $\smash{\deg_G(u)}$ and not by $\smash{\deg_{B_k(G,v)}(u)}$ in \eqref{eq:local-aggregate-trunc} when restricting to the neighborhood $B_k(G,v)$. Recalling \eqref{eq:operator-norm}, the operator $\boldsymbol G^{(v,k)}$ on $\mathcal A^{B_k(G,v)}$ satisfies $\|\boldsymbol G^{(v,k)}\|_{\operatorname{op}}\le 1$. We refer to this truncated local game on $B_k(G,v)$ as $\mathbb{G}^{(v,k)}(\mathcal{A}_{ad}^G,U,\theta,G)$. Under the assumptions of Theorem~\ref{thm:NE}, existence and uniqueness follow as for the game on the original graph $G$. 

\begin{proposition}\label{prop:local-approximation}
Suppose that Assumptions~\ref{assum:U}, \ref{assum:argmax}, and \ref{assum:Av} hold. Let $\bar a\in\mathcal A^G$ be the unique Nash equilibrium of $\mathbb{G}(\mathcal{A}_{ad}^G,U,\theta,G)$, and fix $v\in G$ and $k\in\N_0$. Let $\smash{\bar a^{(v,k)}\in\mathcal A^{B_k(G,v)}}$ be the unique Nash equilibrium of the truncated local game $\mathbb{G}^{(v,k)}(\mathcal{A}_{ad}^G,U,\theta,G)$. Then,
\be\label{eq:local-approx-bound}
\|\bar a_v-\bar a^{(v,k)}_v\|_{\mathcal A}\le 2\rho_{\raisebox{0ex}{\tiny\eqref{eq:rho}}}^{k} M_{\raisebox{-0.3ex}{\tiny\eqref{eq:A_M}}}.
\ee
\end{proposition}

The proof of Proposition~\ref{prop:local-approximation} is given in Section~\ref{sec:proofs-propositions}.

\begin{remark}\label{rem:local-approximation-tradeoff}
Proposition~\ref{prop:local-approximation} quantifies a tradeoff between locality and accuracy. To approximate the global equilibrium action $\bar a_v$ of a given player $v\in G$ within a prescribed tolerance $\delta>0$, it suffices to solve the truncated local game on a ball of radius
\[
k \ge \left\lceil \frac{\log(\delta/(2M_{\raisebox{-0.2ex}{\tiny\eqref{eq:A_M}}}))}{\log\rho_{\raisebox{-0.2ex}{\tiny\eqref{eq:rho}}}}\right\rceil,
\]
since then $2\rho_{\raisebox{0ex}{\tiny\eqref{eq:rho}}}^{k}M_{\raisebox{-0.2ex}{\tiny\eqref{eq:A_M}}}\le \delta$ in \eqref{eq:local-approx-bound}. That is, increasing the radius $k$ enlarges the ball $B_k(G,v)$, but yields an exponentially more accurate local approximation of $\bar a_v$.
\end{remark}
Proposition~\ref{prop:local-approximation} provides a local approximation of the global Nash equilibrium action profile $\bar a\in\mathcal A^G$ in the norm $\|\cdot\|_{\mathcal A}$. For the related notion of $\eps$-Nash equilibria introduced in Definition~\ref{def:eps-Nash}, we are interested in closeness in terms of the players' utility functionals rather than their controls. This is established in Proposition~\ref{prop:epsilon-Nash} below.

\begin{assumption}\label{assum:U-Lip}
Under Assumption~\ref{assum:Av}, assume there exist constants $L_a,L_z\ge 0$ such that for all $\tilde a,\tilde a',\tilde z,\tilde z'\in\mathcal A_{M_{\raisebox{0ex}{\tiny\eqref{eq:A_M}}}}$
and all $\tilde\theta\in\mathcal A$,
\begin{equation}\label{eq:U-Lip}
\big|U(\tilde a,\tilde z,\tilde\theta)-U(\tilde a',\tilde z',\tilde\theta)\big|
\le L_a\|\tilde a-\tilde a'\|_{\mathcal A} + L_z\|\tilde z-\tilde z'\|_{\mathcal A}.
\end{equation}
\end{assumption}

\begin{proposition}\label{prop:epsilon-Nash}
Suppose Assumptions~\ref{assum:U}, \ref{assum:argmax}, \ref{assum:Av}, and \ref{assum:U-Lip} hold. Fix $k\in\N_0$. For each $v\in G$, let $\bar a^{(v,k)}\in\mathcal A^{B_k(G,v)}$ be the unique Nash equilibrium of the truncated local game $\mathbb{G}^{(v,k)}(\mathcal{A}_{ad}^G,U,\theta,G)$ on $B_k(G,v)$, and define the global action profile $\hat a^{(k)}\in\mathcal A^G$ by
\[
\hat a^{(k)}_v := \bar a^{(v,k)}_v,\quad v\in G.
\]
Then $\hat a^{(k)}$ is an $\varepsilon_k$-Nash equilibrium of the global game $\mathbb{G}(\mathcal{A}_{ad}^G,U,\theta,G)$ with
\begin{equation}\label{eq:epsk}
\varepsilon_k := 2\rho_{\raisebox{0ex}{\tiny\eqref{eq:rho}}}^{k} M_{\raisebox{-0.3ex}{\tiny\eqref{eq:A_M}}}(L_a+2L_z).
\end{equation}
In particular, $\varepsilon_k\to 0$ as $k\to\infty$.
\end{proposition}

The proof of Proposition \ref{prop:epsilon-Nash} is given in Section~\ref{sec:proofs-propositions}.

\begin{remark}
Proposition~\ref{prop:epsilon-Nash} shows that a global approximate Nash equilibrium can be constructed from purely local computations. Indeed, for $k\in\N_0$ and each vertex $v\in G$, one solves the truncated local game on the ball $B_k(G,v)$ and uses the resulting local equilibrium action at $v$ to define a global action profile $\hat a^{(k)}$. Under the assumptions of Proposition~\ref{prop:epsilon-Nash}, this profile is an $\varepsilon_k$-Nash equilibrium of the global game with $\eps_k$ as in \eqref{eq:epsk}. In particular, given any target accuracy $\varepsilon>0$, it suffices to choose
\[
k \;\ge\; \left\lceil \frac{\log\!\big(\varepsilon/(2M_{\raisebox{-0.2ex}{\tiny\eqref{eq:A_M}}}(L_a+2L_z))\big)}{\log\rho_{\raisebox{0ex}{\tiny\eqref{eq:rho}}}}\right\rceil,
\]
to ensure that $\hat a^{(k)}$ is an $\varepsilon$-Nash equilibrium. This highlights the same tradeoff between locality and accuracy as addressed in Remark~\ref{rem:local-approximation-tradeoff}.
\end{remark}

\subsection{Local Reconstruction}\label{subsec:local-reconstruction}
The previous subsection showed that equilibrium actions can be approximated from local information. The next result shows that if the Nash equilibrium is known exactly on the vertex boundary of a subgraph, then it can be reconstructed exactly in the interior by running best-response Picard iterates while keeping the boundary actions fixed.

Let $H\subset G$ be an induced subgraph. Define its boundary and interior by
\[
\partial H \,:=\,\Big\{u\in H:\ \exists\,w\notin H \text{ with } w\sim u\Big\},\quad H^\circ \,:=\, H\setminus \partial H.
\]
Note that if $u\in H^\circ$, then every neighbor of $u$ lies in $H$. Next, given $b=(b_u)_{u\in \partial H}\in \mathcal A^{\partial H}$ and $a=(a_u)_{u\in H^\circ}\in\mathcal A^{H^\circ}$, define the clamped profile $a\oplus b\in \mathcal A^{H}$ by
\[
(a\oplus b)_u :=
\begin{cases}
a_u, & u\in H^\circ,\\
b_u, & u\in \partial H.
\end{cases}
\]
For $b\in \mathcal A^{\partial H}$ define the corresponding clamped local aggregate on the interior by
\be\label{eq:local-aggregate-clamped}
z_u^{H,b}(a)
:= \frac{\mathds{1}_{\{\deg_G(u)>0\}}}{\deg_G(u)}\sum_{w\sim u} (a\oplus b)_w,
\quad a\in\mathcal A^{H^\circ},\ u\in H^\circ.
\ee
Moreover, define the clamped best-response operator $\boldsymbol{\Phi}^{H,b}_\theta:\mathcal A^{H^\circ}\to \mathcal A^{H^\circ}$ by
\be\label{eq:best-response-clamped}
\big(\boldsymbol{\Phi}^{H,b}_\theta(a)\big)(u)
:= \argmax_{\tilde a\in \mathcal A_u} U\big(\tilde a,\, z_u^{H,b}(a),\,\theta_u\big),
\quad a\in\mathcal A^{H^\circ},\ u\in H^\circ.
\ee
Starting from $a^{H,b,(0)}\equiv 0$ on $H^\circ$, define the clamped Picard iterates in $\mathcal A^{H^\circ}$ by 
\[
a^{H,b,(k)} := \boldsymbol{\Phi}^{H,b}_\theta\big(a^{H,b,(k-1)}\big),\quad k\ge 1.
\]

\begin{proposition}\label{prop:local-reconstruction}
Suppose that Assumptions~\ref{assum:U}, \ref{assum:argmax}, and \ref{assum:Av} hold. Let $\bar a\in\mathcal A^G$ be the unique Nash equilibrium of $\mathbb{G}(\mathcal{A}_{ad}^G,U,\theta,G)$. Let $H\subset G$ be an induced subgraph and assume that the boundary actions are known exactly, that is, $b=\bar a_{\partial H}\in\mathcal{A}^{\partial H}$ is given. Then:
\begin{itemize}
\item[\textbf{(i)}] The operator $\boldsymbol{\Phi}^{H,b}_\theta$ is a contraction on $\mathcal A^{H^\circ}$ with contraction constant~$\rho_{\raisebox{0ex}{\tiny\eqref{eq:rho}}}<1$ and hence admits a unique fixed point $a^{H,b,\ast}\in\mathcal A^{H^\circ}$.
\item[\textbf{(ii)}] This fixed point coincides with the restriction of the global Nash equilibrium to the interior, that is, $a^{H,b,\ast}=\bar a_{H^\circ}$. In particular, if $v\in H^\circ$, then $\smash{a^{H,b,\ast}_v=\bar a_v}$.
\item[\textbf{(iii)}] The clamped Picard iterates converge geometrically:
\[
\sup_{u\in H^\circ}\big\|a^{H,b,(k)}_u-\bar a_u\big\|_{\mathcal A}\le \rho_{\raisebox{0ex}{\tiny\eqref{eq:rho}}}^{k} M_{\raisebox{-0.3ex}{\tiny\eqref{eq:A_M}}},\quad k\in\N_0.
\]
\end{itemize}
\end{proposition}
The proof of Proposition~\ref{prop:local-reconstruction} is given in Section~\ref{sec:proofs-propositions}.

\begin{figure}[htb]
\centering
\includegraphics[width=0.6\linewidth]{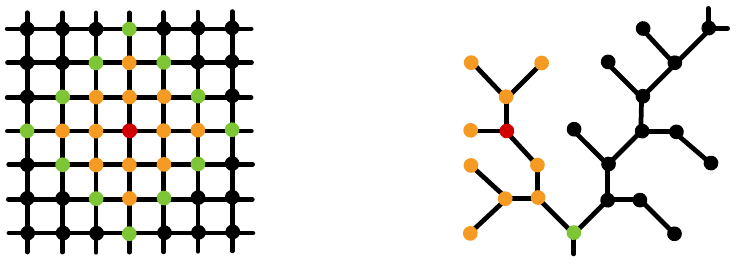}
\caption{Left: the infinite lattice $G\cong\mathbb{Z}^2$ with a specified vertex $v\in G$ (red), the neighborhood $H=B_3(G,v)$ (red/orange/green), its boundary $\partial H$ (green), and its interior $H^\circ$ (red/orange). Right: a tree $G'$ with a specified vertex $v\in G'$ (red), a neighborhood $H'\subset G'$ of $v$ (red/orange/green), its boundary $\partial H'$ (green), and its interior $H'^\circ$ (red/orange).}
\label{fig:local-reconstruction}
\end{figure}

\begin{remark} 
A canonical choice in Proposition~\ref{prop:local-reconstruction} is $H=B_k(G,v)$ for some $k\in\N$. Then $\partial H$ coincides with the outer layer of the ball, i.e.,~vertices in $B_k(G,v)$ that are adjacent to $G\setminus B_k(G,v)$, and it holds that $v\in H^\circ$. Hence, knowing $\bar a$ on this boundary layer allows the exact reconstruction of $\bar a_v$ by solving the clamped fixed point problem on $H^\circ$. The left-hand side of Figure~\ref{fig:local-reconstruction} illustrates this canonical choice in the infinite lattice $G\cong\mathbb{Z}^2$.
\end{remark}

\subsection{Convergence}\label{subsec:convergence}
Next, given a sequence of finite graphs that converges in the local weak sense to a random rooted graph, we study the convergence of the associated Nash equilibria.  The formulation of local weak convergence involves rooted graphs, however,  their root is merely a bookkeeping device and does not endow the corresponding player with any special role in the game. Recall the notion of (marked) local weak convergence from Definition~\ref{def:local-weak-convergence} and Remark~\ref{rem:local-weak-convergence-marked}.

\begin{convention}[Randomness of marked graphs] \label{convention:quenched}
Throughout the rest of this section, we allow both the underlying graphs and their vertex marks to be random. 
We consider two sources of randomness:
\begin{enumerate} 
\item \textbf{Graph randomness.} Let $\P'$ be a probability measure on the space of connected, countable, rooted graphs $\mathcal{G}_\ast$, equipped with the Borel $\sigma$-algebra $\mathcal{B}(\mathcal{G}_\ast)$. A random graph $G \in \mathcal{G}_\ast$ is then distributed according to $\P'$.
\item \textbf{Vertex mark randomness.} Given $G$, each vertex $v \in G$ carries a stochastic process $x^v \in \mathcal A$, where $\mathcal A$ is the Hilbert space defined in \eqref{eq:A}. The family of processes $(x^v)_{v\in G}$ is sampled according to its conditional law given $G$.
\end{enumerate}
The joint law of marked graphs $(G,(x^v)_{v\in G})$ is then denoted by the probability measure $\Q$ on $(\mathcal{G}_\ast[\mathcal A],\mathcal{B}(\mathcal{G}_\ast[\mathcal A]))$, and is obtained by first sampling $G$ according to $\P'$ and then sampling the vertex processes $(x^v)_{v\in G}$ according to their conditional laws given $G$.

Unless stated otherwise, in the following all statements below are understood in the quenched sense:
We fix a realization of the graph under $\P'$ and then the game is analyzed conditional on this realization.  Accordingly, throughout the remainder of this section we treat graphs, the induced games,
and the associated Nash equilibria as random objects with respect to $\P'$, and we do not always indicate this dependence explicitly in the notation for readability. In particular, the terms ``random'' and ``deterministic'' will henceforth refer to randomness with respect to $\P'$.
\end{convention}

\begin{theorem}\label{thm:convergence}
Consider $(G,\theta)\in\mathcal G_\ast[\mathcal{A}]$, that is, a random rooted graph $G\in\mathcal G_\ast$ with marks $(\theta_v)_{v\in G}\in\mathcal{A}^G$. Let $\mathbb{G}(\mathcal{A}_{ad}^G,U,\theta,G)$ be the associated game, and assume that $\mathcal{A}_v=\mathcal{A}_{0}$ for some $\mathcal{A}_{0}\subset \mathcal{A}$ and all~$v\in G$. Suppose that Assumptions~\ref{assum:U}, \ref{assum:argmax}, and \ref{assum:Av} hold, and denote by $\bar a\in\mathcal{A}^G$ the corresponding unique Nash equilibrium. Consider a sequence $\{G_n\}_{n\in\N}\subset\mathcal{G}$ of finite graphs with marks $(\theta^n_v)_{v\in G_n}\in\mathcal{A}^{G_n}$. For $n\in\N$, let $\bar a^n\in\mathcal{A}^{G_n}$ be the  unique Nash equilibrium of the game $\smash{\mathbb{G}(\mathcal{A}_{ad}^{G_n},U,\theta^n,G_n)}$ with $\mathcal{A}_v=\mathcal{A}_{0}$ for all $v\in G_n$. Then, 
\begin{itemize}
    \item[\textbf{(i)}]  If $\{(G_n,\theta^n)\}_{n\in\N}$ converges in distribution in the local weak sense to $(G,\theta)$, then $\{(G_n,\bar a^n)\}_{n\in\N}$ converges in distribution in the local weak sense to $(G,\bar a)$. 
    \item[\textbf{(ii)}] Assume that the Hilbert space $\mathcal{A}$ from \eqref{eq:A} is separable. If $\{(G_n,\theta^n)\}_{n\in\N}$ converges in $\Q$-probability in the local weak sense to $(G,\theta)$ with $|G_n|\to\infty$ in $\Q$-probability, then $\{(G_n,\bar a^n)\}_{n\in\N}$ converges in $\Q$-probability in the local weak sense to $(G,\bar a)$.
\end{itemize}
\end{theorem}

The proof of Theorem~\ref{thm:convergence} is given in Section~\ref{sec:proofs-convergence}.

\begin{remark}\label{rem:convergence-a.s.-subsequence}
The conclusion of Theorem~\ref{thm:convergence}(ii) yields that $\{(G_n,\bar a^n)\}_{n\in\N}$ converges in $\Q$-probability in the local weak sense to $(G,\bar a)$. This implies that there exists a subsequence $\{(G_{n_j},\bar a^{n_j})\}_{j\in\N}$ that converges $\Q$-almost surely in the local weak sense to $(G,\bar a)$.
\end{remark}

\begin{remark}\label{rem:convergence}
In Theorem~\ref{thm:convergence}, we usually have in mind deterministic sequences $\{G_n\}_{n\in\N}$ of finite graphs. Note, however, that the limit object in the statement is a random game, since its underlying graph $G$ is random. This randomness arises by Definition~\ref{def:local-weak-convergence} of local weak convergence, where one studies each finite graph from the perspective of a vertex chosen uniformly at random, and convergence means that the corresponding rooted neighborhoods converge in distribution to the random rooted graph $G$. Consequently, the Nash equilibrium $\bar a$ is random as well.
\end{remark}

\begin{definition}\label{def:separable}
Let $(\Omega,\mathcal F,\P)$ be a probability space. A $\sigma$-algebra $\mathcal{G}\subset\mathcal F$ is called separable if there exists a countable family $\mathcal C\subset\mathcal G$ such that for every $A\in\mathcal G$ and every $\eps>0$ there is $B\in\mathcal C$ with $\P(A\triangle B)<\varepsilon$. In particular, $\mathcal G$ is separable whenever it is countably generated up to $\P$-null sets, i.e.,~whenever there exists a countable family $\mathcal C'\subset\mathcal G$ such that for every $A\in\mathcal G$ there exists $B'\in\sigma(\mathcal C')$ with $\P(A\triangle B')=0$ (see \cite{bruckner1997real}, Example~9.12).
\end{definition}

\begin{remark}\label{rem:separability-A}
The assumption of Theorem~\ref{thm:convergence}(ii) requires the Hilbert space $\mathcal A$ in \eqref{eq:A} to be separable. By Lemma~\ref{lemma:separable-A} below, $\mathcal A$ is separable whenever the filtration $\F=(\mathcal{F}_t)_{0\leq t\leq T}$ is left-continuous and each $\mathcal F_t$ is separable in the sense of Definition~\ref{def:separable}. In particular, this is the case whenever each $\mathcal F_t$ is countably generated up to $\P$-null sets. These statements carry over to the discrete-time case (see Remark~\ref{rem:discrete-time}), that is, $\mathcal A$ is separable whenever each $\sigma$-algebra $\mathcal F_{t_j}$ in \eqref{eq:A-discrete-time} is separable. 
\end{remark}

Theorem~\ref{thm:convergence} assumes that a graph sequence converging in the local weak sense is specified in advance. An alternative is to start from a random rooted graph and construct such a locally weakly convergent graph sequence from it. Recall the Definitions~\ref{def:MTP}, \ref{def:percolation}, and \ref{def:hyperfinite} of hyperfinite unimodular random graphs.

\begin{corollary}\label{cor:convergence}
Consider a hyperfinite unimodular random graph $(G,\theta)=(G,o,\theta)\in\mathcal G_\ast[\mathcal{A}]$ with finitary exhaustion $(H_{\xi_n})_{n\in\N}$ and marks $(\theta_v)_{v\in G}\in\mathcal{A}^G$. Let $\mathbb{G}(\mathcal{A}_{ad}^G,U,\theta,G)$ be the associated game, and assume that $\mathcal{A}_v=\mathcal{A}_{0}$ for some $\mathcal{A}_{0}\subset \mathcal{A}$ and all~$v\in G$. Suppose that Assumptions~\ref{assum:U}, \ref{assum:argmax}, and \ref{assum:Av} hold, and denote by $\bar a\in\mathcal{A}^G$ the corresponding unique Nash equilibrium. For $n\in\N$, define the sequence $\{G_n\}_{n\in\N}\subset\mathcal{G}$ of finite graphs given by
\[
G_n:=[[C_o(H_{\xi_n})]],\quad n\in\N,
\]
with marks $(\theta^n_v)_{v\in G_n}\in\mathcal{A}^{G_n}$ given by the restrictions $\theta^n:=\theta_{G_n}$ for $n\in\N$. Let $\bar a^n\in\mathcal{A}^{G_n}$ be the unique Nash equilibrium of the associated game $\smash{\mathbb{G}(\mathcal{A}_{ad}^{G_n},U,\theta^n,G_n)}$ with $\mathcal{A}_v=\mathcal{A}_{0}$ for all $v\in G_n$. Then, $\{(G_n,\bar a^n)\}_{n\in\N}$ converges in distribution in the local weak sense to $(G,\bar a)$. 
\end{corollary}
The proof of Corollary~\ref{cor:convergence} is given in Section~\ref{sec:proofs-convergence}.

\section{Proofs of the Results in Section \ref{subsec:existence-uniqueness}}\label{sec:proofs-existence-uniqueness}
This section is dedicated to the proof of Theorem~\ref{thm:NE}. Throughout the rest of the paper, we will omit the subscript equation references in the constants $\rho_{\raisebox{0ex}{\tiny\eqref{eq:rho}}}$ from Assumption~\ref{assum:U} and $M_{\raisebox{-0.3ex}{\tiny\eqref{eq:A_M}}}$ from Assumption~\ref{assum:Av} for better readability.

In order to prove Theorem~\ref{thm:NE}, we proceed as in \cite{neuman2025stochastic} and first introduce an auxiliary operator. Namely, recalling \eqref{eq:AG}, given a heterogeneity profile $\theta\in\mathcal{A}^G$, define the best-response operator $\boldsymbol{B}_\theta$ with domain $\mathcal{A}^G$ by
\be \label{eq:B_theta}
(\boldsymbol{B}_\theta z)(v):=\argmax_{\tilde{a}\in \mathcal{A}_v} U\big(\tilde{a}, z_v,\theta_v\big), \quad z\in\mathcal{A}^G,
\ee
which assigns to any fixed local aggregate $z_v$ (not necessarily of the form $z_v(a)$) the best response of player $v\in G$. The argmax in \eqref{eq:B_theta} exists and is unique due to Assumption~\ref{assum:U}. We first prove a result which shows that the image of $\boldsymbol{B}_\theta$ is contained in $\mathcal{A}^G$, turning it into a well-defined operator from $\mathcal{A}^G$ to $\mathcal{A}^G$.
\begin{lemma}\label{lemma:B_theta}
Under Assumptions~\ref{assum:U} and \ref{assum:argmax}, the best-response operator $\boldsymbol{B}_\theta$ satisfies the following:
\begin{itemize}
\item[\textbf{(i)}] $\boldsymbol{B}_\theta$ is jointly Lipschitz continuous, that is,
$$
\left\|\boldsymbol{B}_{\theta}z-\boldsymbol{B}_{\theta'}z'\right\|_{\mathcal{A}^G}\leq\frac{1}{\gamma_U}\big(\ell_U\|z-z'\|_{\mathcal{A}^G}+\ell_\theta\|\theta-\theta'\|_{\mathcal{A}^G}\big),
$$
for all $z,z'\in\mathcal{A}^G$ and $\theta,\theta'\in\mathcal{A}^G$. 
\item[\textbf{(ii)}]  The image of $\boldsymbol{B}_\theta$ is contained in $\mathcal{A}^G$, that is, $\boldsymbol{B}_\theta(\mathcal{A}^G)\subset\mathcal{A}^G$.  
\item[\textbf{(iii)}]  If additionally Assumption~\ref{assum:Av} holds, the image of $\boldsymbol{B}_\theta$ is contained in 
$$\mathcal{A}^G_M:=\left\{a\in\mathcal{A}^G\Big| \|a\|_{\mathcal{A}^G}\leq M\right\}.$$
\end{itemize}
\end{lemma}
\begin{proof}
(i) Recall Assumptions~\ref{assum:U} and \ref{assum:argmax} with the constants introduced therein and \eqref{eq:B_theta}. Let $z,z',\theta,\theta'\in\mathcal{A}^G$. By the same reasoning as in the proof of Lemma 5.2(i) from \cite{neuman2025stochastic}, we obtain through the use of a sensitivity result for variational inequalities that
\be\label{eq:sensitivity}
\|(\boldsymbol{B}_{\theta}z)(v)-(\boldsymbol{B}_{\theta'}z')(v)\|_\mathcal{A}
\leq \frac{1}{\gamma_U}\big(\ell_U\|z_v-z'_v\|_\mathcal{A}+\ell_\theta\|\theta_v-\theta'_v \|_\mathcal{A} \big),\quad \text{for all } v\in G.
\ee
Thus, applying $\sup_{v\in G}$ to both sides of inequality \eqref{eq:sensitivity} and using \eqref{eq:AG} and the triangle inequality completes the proof of (i).

(ii) Let $\tilde z^0, \tilde\theta^0 \in\mathcal{A}$ denote the processes from Assumption~\ref{assum:argmax}. Consider the aggregate $z^0\in\mathcal{A}^G$ defined by $z^0_v:=\tilde z^0$ for all $v\in G$ and the heterogeneity profile $\theta^0\in\mathcal{A}^G$ defined by $\theta^0_v:=\tilde\theta^0$ for all $v\in G$. Then,
$$
\|\boldsymbol{B}_{\theta^0} z^0\|_{\mathcal{A}^G}=\sup_{v\in G} \big\| \argmax_{\tilde{a}\in \mathcal{A}_v} U(\tilde{a}, z^0_v,\theta^0_v)\big\|_\mathcal{A}<\infty.
$$
Now let $z,\theta\in\mathcal{A}^G$. Then it follows from (i) and the triangle inequality,
\be\begin{aligned}
\|\boldsymbol{B}_\theta z\|_{\mathcal{A}^G}&\leq\|\boldsymbol{B}_\theta z-\boldsymbol{B}_{\theta^0} z^0\|_{\mathcal{A}^G}+\|\boldsymbol{B}_{\theta^0} z^0\|_{\mathcal{A}^G}\\
&\leq \frac{\ell_U}{\gamma_U}\| z- z^0\|_{\mathcal{A}^G}+\frac{\ell_\theta}{\gamma_U}\| \theta- \theta^0\|_{\mathcal{A}^G}+\|\boldsymbol{B}_{\theta^0} z^0\|_{\mathcal{A}^G}\\
&\leq \frac{\ell_U}{\gamma_U}\big(\| z\|_{\mathcal{A}^G}+\| z^0\|_{\mathcal{A}^G}\big)+\frac{\ell_\theta}{\gamma_U}\big(\| \theta\|_{\mathcal{A}^G}+\| \theta^0\|_{\mathcal{A}^G}\big)+\|\boldsymbol{B}_{\theta^0} z^0\|_{\mathcal{A}^G}\\
&<\infty,
\end{aligned}\ee
completing the proof of (ii).

(iii) By Assumption~\ref{assum:Av}, there exists  $M>0$ such that $\mathcal{A}_v \subset \mathcal{A}_M$ for all $v\in G$. Let $z\in\mathcal{A}^G$.  Then, by \eqref{eq:B_theta}, it holds that $(\boldsymbol{B}_\theta z)(v)\in\mathcal{A}_M$ for all $v\in G$, and therefore that
$$
\|\boldsymbol{B}_\theta z\|_{\mathcal{A}^G}=\sup_{v\in G}\|(\boldsymbol{B}_\theta z)(v)\|_\mathcal{A}\leq M,
$$
completing the proof of (iii). This concludes the proof of the lemma. 
\end{proof}

We are now set to prove Theorem~\ref{thm:NE}.

\begin{proof}[Proof of Theorem~\ref{thm:NE}] Suppose that Assumptions~\ref{assum:U} and \ref{assum:argmax} hold. For $a\in\mathcal{A}^G$, let $z(a):=(z_v(a))_{v\in G}$ denote the local aggregate profile. Recalling \eqref{eq:AG}, \eqref{eq:local-aggregate}, \eqref{eq:B_theta}, and Definition~\ref{def:Nash}, an action profile $\bar a\in\mathcal A^G$ is a Nash equilibrium of the game $\mathbb{G}(\mathcal{A}_{ad}^G,U,\theta, G)$ if and only if
$$
\bar a_v = \argmax_{\tilde a\in\mathcal A_v} U\big(\tilde a, z_v(\bar a),\theta_v\big)
= (\boldsymbol B_\theta z(\bar a))(v)
= (\boldsymbol B_\theta \boldsymbol G \bar a)(v),\quad \text{for all }v\in G,
$$
that is, if and only if $\bar a = \boldsymbol B_\theta \boldsymbol G \bar a$.
Thus, we need to show that $\boldsymbol B_\theta\boldsymbol G$ has a unique fixed point in $\mathcal A^G$.
By Lemma~\ref{lemma:B_theta}(i), $\boldsymbol B_\theta$ is Lipschitz in its first argument with constant
$\rho=\ell_U/\gamma_U$. Since $\boldsymbol G$ is linear and bounded on $\mathcal{A}^G$ with operator norm $\|\boldsymbol G\|_{\operatorname{op}}\leq 1$ (see Remark~\ref{rem:local-aggregate-operator}), we have for any
$a,a'\in\mathcal A^G$,
\be\label{eq:contraction}
\begin{aligned}
\|\boldsymbol B_\theta\boldsymbol G a
   - \boldsymbol B_\theta\boldsymbol G a'\|_{\mathcal A^G}
&\le \rho
     \|\boldsymbol G a -\boldsymbol G a'\|_{\mathcal A^G}  \\
&\le \rho\,\|\boldsymbol G\|_{\operatorname{op}}\,
       \|a-a'\|_{\mathcal A^G}\\
&\le \rho\,
       \|a-a'\|_{\mathcal A^G}.
\end{aligned}
\ee
Since $\rho<1$ by Assumption~\ref{assum:U}, this shows that
$\boldsymbol B_\theta\boldsymbol G$ is a contraction on
$\mathcal A^G$. By Banach’s fixed point theorem (see \cite{bauschke2017}, Chapter~1.12, Theorem~1.50), it admits a unique fixed point
$\bar a\in\mathcal A^G$, which by construction lies in $\mathcal A_{ad}^G$ and is therefore the unique Nash equilibrium.
\end{proof}

\section{Proofs of the Results in Sections \ref{subsec:correlation-decay}--\ref{subsec:local-reconstruction}}\label{sec:proofs-propositions}
This section focuses on the proofs of Propositions~\ref{prop:correlation-decay}, \ref{prop:local-approximation}, \ref{prop:epsilon-Nash}, and \ref{prop:local-reconstruction}.

\begin{proof}[Proof of Proposition~\ref{prop:correlation-decay}]
Recalling the definition of the local aggregate operator $\boldsymbol{G}$ from \eqref{eq:local-aggregate} and the best-response operator $\boldsymbol B_\theta$ from \eqref{eq:B_theta}, define the sequence of Picard iterates $\{a^{(k)}\}_{k\in\N_0}\subset\mathcal A^G$ by $a^{(0)}\equiv 0$ and
\be\label{eq:PI}
a^{(k)} := \boldsymbol B_\theta \boldsymbol G a^{(k-1)},\quad k\ge 1.
\ee
Recall the notation for subgraphs introduced in Section \ref{subsec:graphs} and before Proposition~\ref{prop:correlation-decay}. By assumption, the processes $(\theta_v)_{v\in G}$ are independent under $\P$ and $\smash{\argmax_{\tilde{a}\in \mathcal{A}_v} U(\tilde{a}, \tilde z,\theta_v)}$ is $\sigma(\tilde z,\theta_v)$-measurable for all $\tilde z\in\mathcal{A}$ and $v\in G$. For $i=1,2$, let $H_i\subset G$ be finite induced subgraphs and let $f_i:\mathcal L^{H_i}\to\R$ be bounded and Lipschitz with respect to $d_{H_i}$ (see \eqref{eq:dH}). Then, it follows inductively from the definition of the iterates \eqref{eq:PI} that, for every $k\geq 1$, the family of strategies $\smash{a^{(k)}_{H_i}}=\smash{(a^{(k)}_v)_{v\in {H_i}}}$ is measurable with respect to $\sigma(\theta_u:\ u\in B_{k-1}(H_i))$ and independent of $\sigma(\theta_u:\ u\notin B_{k-1}(H_i))$. Consequently, if $2k-2< d_G(H_1,H_2)$, then $B_{k-1}(H_1)\cap B_{k-1}(H_2)=\emptyset$, and it follows that $\smash{a^{(k)}_{H_1}}$ and $\smash{a^{(k)}_{H_2}}$ are independent, hence
\be\label{eq:Cov=0}
\mathrm{Cov}_{\P}\big(f_1(a^{(k)}_{H_1}),\,f_2(a^{(k)}_{H_2})\big)=0.
\ee
Next, recalling \eqref{eq:contraction}, it holds that
\[
\|\boldsymbol B_\theta(\boldsymbol G a)-\boldsymbol B_\theta(\boldsymbol G a')\|_{\mathcal A^G}
\le \rho\,\|a-a'\|_{\mathcal A^G},
\]
hence $\boldsymbol B_\theta\boldsymbol G$ is a contraction on $\mathcal{A}^G$ with constant $\rho<1$. Since $\bar a$ is its unique fixed point,
\be\begin{aligned}
\|a^{(k)}-\bar a\|_{\mathcal A^G}
&=\|(\boldsymbol B_\theta\boldsymbol G)^k a^{(0)}-(\boldsymbol B_\theta\boldsymbol G)^k \bar a\|_{\mathcal A^G}\\
&\le \rho^k\,\|a^{(0)}-\bar a\|_{\mathcal A^G},
\end{aligned}\ee
where $(\boldsymbol B_\theta\boldsymbol G)^k$ denotes the $k$-fold composition of the operator $\boldsymbol B_\theta\boldsymbol G$. By Assumption~\ref{assum:Av} and the fact that $a^{(0)}\equiv 0$ it follows that
$\|a^{(0)}-\bar a\|_{\mathcal A^G}\le M$. Therefore,
\begin{equation}\label{eq:geometric-bound-corr}
\sup_{v\in G}\|a^{(k)}_v-\bar a_v\|_{\mathcal A}
\le M\rho^k,\quad k\in\N_0.
\end{equation}
In particular, recalling \eqref{eq:dH}, by the linearity of expectation, the Cauchy-Schwarz inequality, and \eqref{eq:geometric-bound-corr}, it holds for any finite induced subgraph $H\subset G$,
\be\begin{aligned}\label{eq:dH-bound}
\E_{\P}\big[d_H(a^{(k)}_H,\bar a_H)\big]&\leq\E_{\P}\Big[\sum_{v\in H}\|a^{(k)}_v-\bar a_v\|_{L^2}\Big]\\
&=\sum_{v\in H}\E_{\P}\big[\|a^{(k)}_v-\bar a_v\|_{L^2}\big]\\
&\leq\sum_{v\in H}\|a^{(k)}_v-\bar a_v\|_{\mathcal A}\\
&\leq |H|\big(\sup_{v\in H}\|a^{(k)}_v-\bar a_v\|_{\mathcal A}\big)\\
&\le |H|M\rho^k.
\end{aligned}\ee
Now, fix finite induced subgraphs $H_1,H_2\subset G$ and choose $k:=\lceil d_G(H_1,H_2)/2\rceil$, so that $2k-2< d_G(H_1,H_2)$.
Set $\smash{\Delta_i:=f_i(\bar a_{H_i})-f_i(a^{(k)}_{H_i})}$ for $i=1,2$. Then, by the linearity of covariance and \eqref{eq:Cov=0},
\be\begin{aligned}\label{eq:Cov}
\mathrm{Cov}_{\P}\big(f_1(\bar a_{H_1}),f_2(\bar a_{H_2})\big)
&= \mathrm{Cov}_{\P}\big(\Delta_1,f_2(\bar a_{H_2})\big)+\mathrm{Cov}_{\P}\big(f_1(a^{(k)}_{H_1}),f_2(\bar a_{H_2})\big)\\
&=\mathrm{Cov}_{\P}\big(\Delta_1,f_2(\bar a_{H_2})\big)+\mathrm{Cov}_{\P}\big(f_1(a^{(k)}_{H_1}),\Delta_2\big).
\end{aligned}\ee
We thus obtain from \eqref{eq:Cov} that
\be\label{eq:Cov-ineq1}
\Big|\mathrm{Cov}_{\P}\big(f_1(\bar a_{H_1}),f_2(\bar a_{H_2})\big)\Big|
\le
2\|f_2\|_\infty\,\E_\P[|\Delta_1|]
+2\|f_1\|_\infty\,\E_\P[|\Delta_2|].
\ee
By the Lipschitz continuity of the $f_i$ and \eqref{eq:dH-bound}, it holds that
\be\label{eq:Cov-ineq2}
\E_\P[|\Delta_i|]
\le \E_\P\big[\mathrm{Lip}(f_i)\, d_{H_i}(\bar a_{H_i},a^{(k)}_{H_i})\big]
\le M\rho^k\,|H_i|\,\mathrm{Lip}(f_i),
\quad i=1,2.
\ee
Finally, combining \eqref{eq:Cov-ineq1} and \eqref{eq:Cov-ineq2} and recalling \eqref{eq:BL-norm} yields
\be\begin{aligned}
\Big|\mathrm{Cov}_{\P}\big(f_1(\bar a_{H_1}),f_2(\bar a_{H_2})\big)\Big|
&\le 2M\,\rho^k\Big(\|f_2\|_\infty\,\mathrm{Lip}(f_1)|H_1| +\|f_1\|_\infty\,\mathrm{Lip}(f_2)|H_2|\Big)\\
&\le 2M\,\rho^k\,(|H_1|+|H_2|)\|f_1\|_{\text{BL}}\,\|f_2\|_{\text{BL}},
\end{aligned}\ee
as claimed.
\end{proof}

\begin{proof}[Proof of Proposition~\ref{prop:local-approximation}] Fix $v\in G$ and $k\in\N_0$.
Analogously to \eqref{eq:B_theta}, define the truncated best-response operator
$\smash{\boldsymbol B^{(v,k)}_\theta}$ on $\smash{\mathcal A^{B_k(G,v)}}$ by
\be
(\boldsymbol B^{(v,k)}_\theta z)(u):=\argmax_{\tilde a\in \mathcal A_u} U(\tilde a,z_u,\theta_u),
\quad z=(z_u)_{u\in B_k(G,v)}\in \mathcal A^{B_k(G,v)},\ u\in B_k(G,v).
\ee
Then, recalling \eqref{eq:local-aggregate-trunc}, the unique Nash equilibrium $\bar a^{(v,k)}\in \mathcal A^{B_k(G,v)}$ of the truncated local game $\mathbb{G}^{(v,k)}(\mathcal{A}_{ad}^G,U,\theta,G)$ on $B_k(G,v)$ introduced before Proposition~\ref{prop:local-approximation} is given by the unique fixed point
\be\label{eq:NE-local}
\bar a^{(v,k)}=\boldsymbol B^{(v,k)}_\theta \boldsymbol G^{(v,k)}\bar a^{(v,k)}.
\ee
By \eqref{eq:contraction}, $\boldsymbol B_\theta\boldsymbol G$ is a contraction on $\mathcal A^G$ with constant $\rho<1$. An analogous argument using the fact that $\|\boldsymbol G^{(v,k)}\|_{\operatorname{op}}\le 1$ shows that $\smash{\boldsymbol B^{(v,k)}_\theta\boldsymbol G^{(v,k)}}$ is a contraction on $\mathcal A^{B_k(G,v)}$ with constant $\rho<1$ as well. Recall the definition of the global Picard iterates $\{a^{(m)}\}_{m\in\N_0}$ from \eqref{eq:PI}. Similarly, define the local Picard iterates on $B_k(G,v)$ by $a^{(v,k),(0)}\equiv 0$ and
\[
a^{(v,k),(m)}:=\boldsymbol B^{(v,k)}_\theta\boldsymbol G^{(v,k)}a^{(v,k),(m-1)},\quad m\ge 1.
\]
Then it holds for every $m\leq k$ and every $u\in B_{k-m}(G,v)$ that the associated global and local iterate coincide, that is,
\be
a^{(m)}_u=a^{(v,k),(m)}_u.
\ee
Namely, for $m=0$ the claim is immediate since both iterates are equal to $0$, and since the local iterate $\smash{a^{(m)}_u}$ at each vertex $u$ only depends on the local iterates $\smash{a^{(m-1)}_w}$ with $w\in B_1(G,u)$, the claim follows inductively for all $m\leq k$.
In particular, taking $u=v$ yields
\be\label{eq:iterates-agree}
a^{(m)}_v=a^{(v,k),(m)}_v,\quad \text{for all }m\le k.
\ee
Next, as $\bar a$ and $\bar a^{(v,k)}$ are the unique fixed points of
$\boldsymbol B_\theta\boldsymbol G$ and $\boldsymbol B^{(v,k)}_\theta\boldsymbol G^{(v,k)}$,
respectively, the contraction property yields as in \eqref{eq:geometric-bound-corr} that
\be\label{eq:global-local-bounds}
\|a^{(k)}_v-\bar a_v\|_{\mathcal A}\le M\rho^k,
\quad
\|a^{(v,k),(k)}_v-\bar a^{(v,k)}_v\|_{\mathcal A}\le M\rho^k.
\ee
Finally, combining \eqref{eq:iterates-agree} and \eqref{eq:global-local-bounds} and applying the triangle inequality yields
\be\begin{aligned}
\|\bar a_v-\bar a^{(v,k)}_v\|_{\mathcal A}
&\le \|\bar a_v-a^{(k)}_v\|_{\mathcal A}
   +\|a^{(v,k),(k)}_v-\bar a^{(v,k)}_v\|_{\mathcal A}\\
&\le 2M\rho^k,
\end{aligned}\ee
completing the proof.
\end{proof}

\begin{proof}[Proof of Proposition~\ref{prop:epsilon-Nash}] Recall Definition \ref{def:eps-Nash} of $\eps$-Nash equilibria. Let $\bar a\in\mathcal A^G$ be the unique Nash equilibrium of the global game $\mathbb{G}(\mathcal{A}_{ad}^G,U,\theta,G)$. Fix $k\in\N_0$. For each $v\in G$, let $\bar a^{(v,k)}\in\mathcal A^{B_k(G,v)}$ be the unique Nash equilibrium of the truncated local game $\mathbb{G}^{(v,k)}(\mathcal{A}_{ad}^G,U,\theta,G)$ on $B_k(G,v)$, and define the global action profile $\hat a^{(k)}\in\mathcal A^G$ constructed through local game equilibria by
\[
\hat a^{(k)}_v := \bar a^{(v,k)}_v,\quad v\in G.
\]
By Proposition~\ref{prop:local-approximation} applied at each vertex $u\in G$,
\begin{equation}\label{eq:sup-close-Lip}
\sup_{u\in G}\|\hat a^{(k)}_u-\bar a_u\|_{\mathcal A}\le 2M\rho^k.
\end{equation}
Fix $v\in G$ and set $\hat z^{(k)}_v:=z_v(\hat a^{(k)})$ and $\bar z_v:=z_v(\bar a)$.
Then, since $z_v(\cdot)$ is an average over neighbors and by \eqref{eq:sup-close-Lip},
\begin{equation}\begin{aligned}\label{eq:agg-close-Lip}
\|\hat z_v^{(k)}-\bar z_v\|_{\mathcal{A}}&\le \sup_{u\sim v}\|\hat a^{(k)}_u-\bar a_u\|_{\mathcal A}\\
&\le 2M\rho^k.
\end{aligned}
\end{equation}
Let $\hat a_v\in\mathcal A_v$ be such that
\be\label{eq:hat-av}
U(\hat a_v,\hat z_v^{(k)},\theta_v)=\sup_{\tilde a\in\mathcal A_v}U(\tilde a,\hat z_v^{(k)},\theta_v),
\ee
which exists by Assumption~\ref{assum:U}.
Then
\be\label{eq:sup-equality}
\sup_{\tilde a\in\mathcal A_v}U(\tilde a,\hat z^{(k)}_v,\theta_v)-U(\hat a^{(k)}_v,\hat z^{(k)}_v,\theta_v)=U(\hat a_v,\hat z^{(k)}_v,\theta_v)-U(\hat a^{(k)}_v,\hat z^{(k)}_v,\theta_v).
\ee
Now
\be\begin{aligned}\label{eq:split-difference}
U(\hat a_v,\hat z^{(k)}_v,\theta_v)-U(\hat a^{(k)}_v,\hat z^{(k)}_v,\theta_v)
&= \big(U(\hat a_v,\hat z^{(k)}_v,\theta_v)-U(\hat a_v,\bar z_v,\theta_v)\big)\\
&\quad + \big(U(\hat a_v,\bar z_v,\theta_v)-U(\bar a_v,\bar z_v,\theta_v)\big)\\
&\quad +\big(U(\bar a_v,\bar z_v,\theta_v)-U(\hat a^{(k)}_v,\hat z^{(k)}_v,\theta_v)\big).
\end{aligned}\ee
The second term of the right-hand side of \eqref{eq:split-difference} is $\leq 0$ by property \eqref{eq:Nash} of $\bar a_v$. Its third term can be bound using the triangle inequality,
\be\begin{aligned}\label{eq:third-term}
\big|U(\bar a_v,\bar z_v,\theta_v)-U(\hat a^{(k)}_v,\hat z^{(k)}_v,\theta_v)\big|
\le& \big|U(\bar a_v,\bar z_v,\theta_v)-U(\hat a^{(k)}_v,\bar z_v,\theta_v)\big|\\
&+\big|U(\hat a^{(k)}_v,\bar z_v,\theta_v)-U(\hat a^{(k)}_v,\hat z^{(k)}_v,\theta_v)\big|.
\end{aligned}\ee
Applying \eqref{eq:U-Lip} to the first term of the right-hand side of \eqref{eq:split-difference} and the two terms on the right-hand side of \eqref{eq:third-term} yields
\be\label{eq:Lz-La-bound}
U(\hat a_v,\hat z^{(k)}_v,\theta_v)-U(\hat a^{(k)}_v,\hat z^{(k)}_v,\theta_v)
\le 2L_z\|\hat z^{(k)}_v-\bar z_v\|_{\mathcal A}+L_a\|\hat a^{(k)}_v-\bar a_v\|_{\mathcal A}.
\ee
Combining \eqref{eq:sup-close-Lip}, \eqref{eq:agg-close-Lip}, \eqref{eq:sup-equality}, and \eqref{eq:Lz-La-bound} gives
\[
\sup_{\tilde a\in\mathcal A_v}U(\tilde a,z_v(\hat a^{(k)}),\theta_v)
-U(\hat a^{(k)}_v,z_v(\hat a^{(k)}),\theta_v)
\le 2M\rho^k\,(L_a+2L_z).
\]
Since $v\in G$ was arbitrary, this is exactly \eqref{eq:eps-Nash} with $\varepsilon=\varepsilon_k$.
\end{proof}

We conclude this section with the proof of Proposition~\ref{prop:local-reconstruction}.

\begin{proof}[Proof of Proposition~\ref{prop:local-reconstruction}]  Recall the definitions and notation introduced before Proposition~\ref{prop:local-reconstruction}. Let $\bar a\in\mathcal A^G$ be the unique Nash equilibrium of $\mathbb{G}(\mathcal{A}_{ad}^G,U,\theta,G)$.
Let $H\subset G$ be an induced subgraph and assume that the boundary actions are known exactly,
that is, $b=\bar a_{\partial H}\in\mathcal{A}^{\partial H}$ is given.

(i) We first show that the clamped best-response operator $\boldsymbol{\Phi}^{H,b}_\theta$ from \eqref{eq:best-response-clamped} is a contraction on $\mathcal{A}^{H^\circ}$.
Let $a,a'\in\mathcal A^{H^\circ}$ and extend them to clamped profiles $a\oplus b$ and $a'\oplus b$ on $H$.
For any $u\in H^\circ$, using that the averaging weights in \eqref{eq:local-aggregate-clamped} sum to $1$ and that only interior coordinates may differ,
\be\begin{aligned}\label{eq:clamped-agg-Lip}
\big\|z_u^{H,b}(a)-z_u^{H,b}(a')\big\|_{\mathcal A}
&\le \sup_{w\in H^\circ}\|a_w-a'_w\|_{\mathcal A}\\
&= \|a-a'\|_{\mathcal A^{H^\circ}}.
\end{aligned}\ee
Next, similarly to \eqref{eq:B_theta}, define the restricted best-response operator  $\boldsymbol{B}_\theta^{H}$ with domain $\mathcal{A}^{H^\circ}$ by
\be \label{eq:B_theta-restricted}
(\boldsymbol{B}_\theta^{H}z)(u):=\argmax_{\tilde{a}\in \mathcal{A}_u} U\big(\tilde{a}, z_u,\theta_u\big), \quad z=(z_u)_{u\in H^\circ}\in\mathcal{A}^{H^\circ},
\ee
Define the aggregate profile $z^{H,b}(a):=(z^{H,b}_u(a))_{u\in H^\circ}\in\mathcal A^{H^\circ}$. Then, recalling \eqref{eq:best-response-clamped}, we have that 
$$
\big(\boldsymbol{\Phi}^{H,b}_\theta(a)\big)(u)=\big(\boldsymbol{B}^{H}_{\theta}z^{H,b}(a)\big)(u),\quad \text{for all } u\in H^\circ,
$$ 
and an analogous equality holds for $a'$. Therefore, recalling $\rho=\ell_U/\gamma_U$ and applying the pointwise sensitivity estimate \eqref{eq:sensitivity} on $H^\circ$, we obtain that
\be\label{eq:clamped-sensitivity}
\big\|\big(\boldsymbol{\Phi}^{H,b}_\theta(a)\big)(u)-\big(\boldsymbol{\Phi}^{H,b}_\theta(a')\big)(u)\big\|_{\mathcal A}
\le \rho\,\big\|z_u^{H,b}(a)-z_u^{H,b}(a')\big\|_{\mathcal A},\quad \text{for all } u\in H^\circ.
\ee
Combining \eqref{eq:clamped-agg-Lip} and \eqref{eq:clamped-sensitivity} and taking the supremum over all $u\in H^\circ$ thus yields
\be\label{eq:contraction-clamped}
\|\boldsymbol{\Phi}^{H,b}_\theta(a)-\boldsymbol{\Phi}^{H,b}_\theta(a')\|_{\mathcal A^{H^\circ}}
\le \rho\,\|a-a'\|_{\mathcal A^{H^\circ}},
\ee
so $\boldsymbol{\Phi}^{H,b}_\theta$ is a contraction with constant $\rho<1$. This proves (i).

(ii) Next, we verify that $\bar a_{H^\circ}=(\bar a_u)_{u\in H^\circ}\in\mathcal A^{H^\circ}$ is a fixed point when $b=\bar a_{\partial H}$.
Let $u\in H^\circ$, then, since $u$ has no neighbors outside $H$, recalling \eqref{eq:local-aggregate} and \eqref{eq:local-aggregate-clamped}, we have
\be\label{eq:aggregates-agree}\begin{aligned}
z_u(\bar a)&=\frac{\mathds{1}_{\{\deg_G(u)>0\}}}{\deg_G(u)}\sum_{w\sim u}\bar a_w\\
&=\frac{\mathds{1}_{\{\deg_G(u)>0\}}}{\deg_G(u)}\sum_{w\sim u}(\bar a_{H^\circ}\oplus \bar a_{\partial H})_w\\
&= z_u^{H,\bar a_{\partial H}}(\bar a_{H^\circ}).
\end{aligned}\ee
Now, using the Nash equilibrium condition \eqref{eq:Nash} for $u$ and \eqref{eq:aggregates-agree} yields
\be\begin{aligned}
\bar a_u&=\argmax_{\tilde a\in\mathcal A_u}U\big(\tilde a, z_u(\bar a),\theta_u\big)\\
&=\argmax_{\tilde a\in\mathcal A_u}U\big(\tilde a, z_u^{H,\bar a_{\partial H}}(\bar a_{H^\circ}),\theta_u\big)\\
&=\big(\boldsymbol{\Phi}^{H,\bar a_{\partial H}}_\theta(\bar a_{H^\circ})\big)(u).
\end{aligned}\ee
Thus $\boldsymbol{\Phi}^{H,b}_\theta(\bar a_{H^\circ})=\bar a_{H^\circ}$, so $\bar a_{H^\circ}$ is a fixed point of $\boldsymbol{\Phi}^{H,b}_\theta$.
By the uniqueness of the fixed point from (i), this proves (ii).

(iii) Finally, the contraction property \eqref{eq:contraction-clamped} and Assumption~\ref{assum:Av} imply for all $k\in\N_0$,
\be\begin{aligned}
\|a^{H,b,(k)}-\bar a_{H^\circ}\|_{\mathcal A^{H^\circ}}
&\le \rho^k\,\|a^{H,b,(0)}-\bar a_{H^\circ}\|_{\mathcal A^{H^\circ}}\\
&\le \rho^k\,M,
\end{aligned}\ee
completing the proof of (iii).
\end{proof}

\section{Proofs of the Results in Section~\ref{subsec:convergence}}\label{sec:proofs-convergence}

This section is dedicated to the proofs of Theorem~\ref{thm:convergence} and Corollary~\ref{cor:convergence}. For this, several auxiliary lemmas are needed. Recall Definition~\ref{def:local-convergence-marked} of marked local convergence and the metric $d_\ast$ from \eqref{eq:d-ast-marked}.

\begin{lemma}\label{lemma:map1}
The map 
\be
\mathcal G_\ast[\mathcal{A}]\to \mathcal G_\ast[\mathcal{A}],\quad (G,(a_v)_{v\in G})\mapsto (G, (z_v)_{v\in G}),\quad z_v:=\frac{\mathds{1}_{\{\deg_G(v)>0\}}}{\deg_G(v)}\sum_{u\sim v} a_u,\quad v\in G,
\ee
is continuous with respect to the metric $d_\ast$.
\end{lemma}
\begin{proof}
Assume that $\{(G_n,a^n)\}_{n\in\N}$ converges locally to $(G,a)$ in $\mathcal G_\ast[\mathcal{A}]$. That is, for every $k\in\N$ and $\eps>0$, there exists $n_{k,\eps}\in\N$ such that for all $n\geq n_{k,\eps}$ there is an isomorphism $\varphi:B_k(G_n)\to B_k(G)$ with 
$$\max_{v\in B_k(G_n)}\|a_v^n-a_{\varphi(v)}\|_{\mathcal A}<\eps.$$ 
Now, given $k'\in \N$ and $\eps'>0$, choose $k=k'+1$ and $\eps=\eps'$. Then, $\varphi$ is an isomorphism $B_{k'+1}(G_n)\to B_{k'+1}(G)$ that satisfies 
$$\max_{u\in B_{k'+1}(G_n)}\|a_u^n-a_{\varphi(u)}\|_{\mathcal A}<\eps'.$$ 
Therefore, since the degree function is invariant under isomorphisms and by the triangle inequality,
\be\begin{aligned}
\max_{v\in B_{k'}(G_n)}\|z_v^n-z_{\varphi(v)}\|_{\mathcal A}&=\max_{v\in B_{k'}(G_n)}\big\|\frac{\mathds{1}_{\{\deg_{G_n}(v)>0\}}}{\deg_{G_n}(v)}\sum_{u\sim v}a^n_u-\frac{\mathds{1}_{\{\deg_{G}(\varphi(v))>0\}}}{\deg_{G}(\varphi(v))}\hspace{-1mm}\sum_{\varphi(u)\sim \varphi(v)}a_{\varphi(u)}\big\|_{\mathcal A}\\
&\leq\max_{v\in B_{k'}(G_n)}\frac{\mathds{1}_{\{\deg_{G_n}(v)>0\}}}{\deg_{G_n}(v)}\sum_{u\sim v}\big\|(a^n_u-a_{\varphi(u)})\big\|_{\mathcal A}< \eps'.
\end{aligned}\ee
That is, $\{(G_n,z^n)\}_{n\in\N}$ converges locally to $(G,z)$ in $\mathcal G_\ast[\mathcal{A}]$. As $d_\ast$ metrizes local convergence, this implies continuity.
\end{proof}
Recall that in Theorem~\ref{thm:convergence} the players are assumed to have homogeneous action sets, that is, $\mathcal{A}_v=\mathcal{A}_{0}$ for some $\mathcal{A}_{0}\subset \mathcal A$ and all~$v$. Define the corresponding best-response map
\be\label{eq:beta}
\beta:\mathcal{A}\times\mathcal{A}\to\mathcal{A},\quad (\tilde z,\tilde \theta)\mapsto \beta_{\tilde\theta}(\tilde z):= \argmax_{\tilde {a}\in\mathcal{A}_0} \ U\big(\tilde{a}, \tilde z,\tilde \theta\big),
\ee
that assigns the best-response to a given local aggregate $\tilde z$ and a heterogeneity process $\tilde\theta$. Under Assumption~\ref{assum:U}, the argmax in \eqref{eq:beta} always exists and $\beta$ is well-defined.
\begin{lemma}\label{lemma:beta-Lipschitz}
Under Assumptions~\ref{assum:U} and \ref{assum:argmax}, the map $\beta$ is jointly Lipschitz continuous, that is,
$$
\left\|\beta_{\tilde\theta}(\tilde z)-\beta_{\tilde\theta'}(\tilde z')\right\|_{\mathcal{A}}\leq\frac{1}{\gamma_U}\big(\ell_U\|\tilde z-\tilde z'\|_{\mathcal{A}}+\ell_\theta\|\tilde\theta-\tilde\theta'\|_{\mathcal{A}}\big),
$$
for all $\tilde z,\tilde z',\tilde \theta,\tilde \theta'\in\mathcal{A}$. 
\end{lemma}
\begin{proof}
    The proof follows directly from setting $z_v:=\tilde z$, $z'_v:=\tilde z'$, $\theta_v:=\tilde\theta$, $\theta'_v:=\tilde\theta'$ for all $v\in G$ and applying Lemma \ref{lemma:B_theta}(i).
\end{proof}
\begin{lemma}\label{lemma:map2}
The map 
\be
\mathcal G_\ast[\mathcal{A}\times\mathcal{A}]\to \mathcal G_\ast[\mathcal{A}\times\mathcal{A}],\quad \big(G,(z_v)_{v\in G},(\theta_v)_{v\in G}\big)\mapsto \big(G, \big(\beta_{\theta_v}(z_v)\big)_{v\in G},(\theta_v)_{v\in G}\big),
\ee
is continuous with respect to the metric $d_\ast$.
\end{lemma}
\begin{proof}
This follows by Definition~\ref{def:local-convergence-marked} of marked local convergence and the global Lipschitz property of $\beta$ from Lemma \ref{lemma:beta-Lipschitz}. 
\end{proof}
Combining Lemma \ref{lemma:map1} and \ref{lemma:map2} yields the following useful result.
\begin{lemma}\label{lemma:B-continuous}
The map
\be\begin{aligned}\label{eq:B}
B:\mathcal G_\ast[\mathcal{A}\times\mathcal{A}]&\to \mathcal G_\ast[\mathcal{A}\times\mathcal{A}],\\
(G,(a_v)_{v\in G},(\theta_v)_{v\in G})&\mapsto \big(G, \big(\beta_{\theta_v}(\frac{\mathds{1}_{\{\deg_G(v)>0\}}}{\deg_G(v)}\sum_{u\sim v}a_u)\big)_{v\in G},(\theta_v)_{v\in G}\big),
\end{aligned}\ee
is continuous with respect to the metric $d_\ast$.
\end{lemma}
\begin{proof}
By Lemma \ref{lemma:map1} and \ref{lemma:map2}, $B$ is given by a composition of continuous maps.
\end{proof}

The final ingredient we need for the proof of Theorem~\ref{thm:convergence} is the following statement.
\begin{lemma}\label{lemma:aux}
Let $(\mathcal X,d)$ be a metric space. Let $\{X_n\}_{n\in\N}$ and $X$ be $\mathcal X$-valued random variables on a probability space $(\tilde \Omega,\tilde{\mathcal{F}},\mu)$. For any $\eps>0$, assume that there are $\mathcal X$-valued random variables  $\{X^\eps_n\}_{n\in\N}$ and $X^\eps$ such that 
\begin{enumerate}
    \item $d(X_n,X_n^\eps)\leq \eps$ for all $n\in \N$ and $d(X,X^\eps)\leq \eps$, $\mu$-almost surely. 
    \item $(X_n^\eps)_{n\in\N}$ converges in distribution to $X^\eps$.
\end{enumerate}
Then $\{X_n\}_{n\in\N}$ converges in distribution to $X$.
\end{lemma}
\begin{proof}
Let $K\subset\mathcal X$ be closed. For $\delta>0$ define the closed sets
$$
K^\delta := \{x\in\mathcal X:\ d(x,K)\le \delta\}.
$$
Now $d(X_n,X_n^\eps)\leq \eps$ implies $\{X_n\in K\}\subset \{X_n^\eps\in K^\eps\}$ for all $n\in\N$, and thus
\begin{equation}\label{eq:aux-prob-1}
\mu(X_n\in K)\le \mu(X_n^\eps\in K^\eps),\quad n\in\N.
\end{equation}
By assumption, $X_n^\eps\Rightarrow X^\eps$. Since $K^\eps$ is closed, the Portmanteau lemma yields
\begin{equation}\label{eq:aux-portmanteau}
\limsup_{n\to\infty}\mu(X_n^\eps\in K^\eps)\le \mu(X^\eps\in K^\eps).
\end{equation}
Combining \eqref{eq:aux-prob-1} and \eqref{eq:aux-portmanteau} gives
\begin{equation}\label{eq:aux-prob-2}
\limsup_{n\to\infty}\mu(X_n\in K)\le \mu(X^\eps\in K^\eps).
\end{equation}
Next, $d(X,X^\eps)\le \eps$ implies that $\{X^\eps\in K^\eps\}\subset \{X\in K^{2\eps}\}$, and thus
\begin{equation}\label{eq:aux-prob-3}
\mu(X^\eps\in K^\eps)\le \mu(X\in K^{2\eps}).
\end{equation}
Combining \eqref{eq:aux-prob-2} and \eqref{eq:aux-prob-3} yields
\begin{equation}\label{eq:aux-prob-4}
\limsup_{n\to\infty}\mu(X_n\in K)\le \mu(X\in K^{2\eps}).
\end{equation}
Since $K$ is closed, we have $K^{2\eps}\downarrow K$ as $\eps\downarrow 0$. By continuity $\mu$,
$$
\lim_{\eps\downarrow 0}\mu(X\in K^{2\eps})
=\mu(X\in K).
$$
Thus taking the limit over $\eps\downarrow 0$ in \eqref{eq:aux-prob-4} we conclude that
\[
\limsup_{n\to\infty}\mu(X_n\in K)\le \mu(X\in K).
\]
By the Portmanteau lemma, $X_n\Rightarrow X$, as claimed.
\end{proof}

We are now ready to prove Theorem \ref{thm:convergence}. Note that we have only worked with local convergence of $\mathcal{A}$-marked rooted graphs (see Definition~\ref{def:local-convergence-marked}) so far. Now we will also consider local weak convergence of $\mathcal{A}$-marked unrooted graphs (see Definition~\ref{def:local-weak-convergence} and Remark~\ref{rem:local-weak-convergence-marked}).

\begin{proof}[Proof of Theorem \ref{thm:convergence}] Recall Convention~\ref{convention:quenched} and suppose that the assumptions of Theorem~\ref{thm:convergence} hold.
Given a random rooted graph $G\in \mathcal{G}_\ast$, let 
$$(G,(a^{(0)}_v)_{v\in G},(\theta_v)_{v\in G})\in \mathcal G_\ast[\mathcal{A}\times \mathcal{A}]$$
be defined by $a^{(0)}_v=0$ for all $v\in G$. Similar to the proof of Proposition~\ref{prop:correlation-decay}, inductively define the Picard iterates in $\mathcal G_\ast[\mathcal{A}\times\mathcal{A}]$ given by
\be\label{eq:Picard-iterates}
(G,(a^{(k)}_v)_{v\in G},(\theta_v)_{v\in G}):=B\big(G,(a^{(k-1)}_v)_{v\in G},(\theta_v)_{v\in G}\big), \quad k\geq 1,
\ee
where $B$ is defined in \eqref{eq:B}. Recall that $\rho=\ell_U/\gamma_U$. Then, by Lemma \ref{lemma:beta-Lipschitz}, it holds for all $v\in G$ that 
\be\begin{aligned}
\|a^{(k)}_v-\bar a_v\|_{\mathcal A}&=\| \beta_{\theta_v}(\frac{\mathds{1}_{\{\deg_G(v)>0\}}}{\deg_G(v)}\sum_{u\sim v}a^{(k-1)}_u) - \beta_{\theta_v}(\frac{\mathds{1}_{\{\deg_G(v)>0\}}}{\deg_G(v)}\sum_{u\sim v}\bar a_u)\|_{\mathcal A}\\
&\leq \rho\ \|\frac{\mathds{1}_{\{\deg_G(v)>0\}}}{\deg_G(v)}\sum_{u\sim v}(a^{(k-1)}_u-\bar a_u)\|_{\mathcal A}\\
&\leq\rho\ \sup_{u\in G}\|a^{(k-1)}_u-\bar a_u\|_{\mathcal A},
\end{aligned}\ee
and thus by iteration and Assumption \ref{assum:Av} for all $v\in G$ that 
\be\begin{aligned}\label{eq:geometric-bound}
\|a^{(k)}_v-\bar a_v\|_{\mathcal A}&\leq \rho^k \sup_{u\in G}\|a^{(0)}_u-\bar a_u\|_{\mathcal A}\\
&= \rho^k \sup_{u\in G}\|\bar a_u\|_{\mathcal A}\\
&\leq \rho^kM.
\end{aligned}\ee
Now consider a sequence $\{G_n\}_{n\in\N}\subset\mathcal{G}$ of finite graphs with marks $(\theta^n_v)_{v\in G_n}\in\mathcal{A}^{G_n}$ and let $(G,\theta)\in \mathcal G_\ast[\mathcal{A}]$. For $n\in\N$, let $\bar a^n\in\mathcal{A}^{G_n}$ be the  unique Nash equilibrium of the game $\smash{\mathbb{G}(\mathcal{A}_{ad}^{G_n},U,\theta^n,G_n)}$ with $\mathcal{A}_v=\mathcal{A}_{0}$ for all $v\in G_n$, and let $\bar a\in\mathcal{A}^{G}$ be the unique Nash equilibrium of the game $\smash{\mathbb{G}(\mathcal{A}_{ad}^{G},U,\theta,G)}$ with $\mathcal{A}_v=\mathcal{A}_{0}$ for all $v\in G$. 

(i) First, assume that $\{(G_n,\theta^n)\}_{n\in\N}$ converges in distribution in the local weak sense to $(G,\theta)$. By Remark~\ref{rem:local-weak-convergence-marked}, this means that 
\be\label{eq:conv-theta-marks}
\E_{\Q}\Big[\frac{1}{|G_n|}\sum_{v \in G_n} h\big(C_v(G_n,\theta^n)\big)\Big]
\;\xrightarrow[n\to\infty]\;
\E_{\Q}\big[h(G,\theta)\big],
\quad \text{for all } h \in C_b\big(\mathcal G_\ast[\mathcal A]\big).
\ee
Let $\smash{a^{(0),n}_v=0}$ for all $v\in G_n$ and $n\in\N$, and $a^{(0)}_v=0$ for all $v\in G$. Let $\iota:\mathcal G_\ast[\mathcal A]\to \mathcal G_\ast[\mathcal A\times\mathcal{A}]$ denote the (continuous) inclusion adding a first coordinate equal to zero. Then setting $h:=\hat h\circ\iota$ in \eqref{eq:conv-theta-marks} implies
\be\label{eq:conv-a0-theta-marks}
\E_{\Q}\Big[\frac{1}{|G_n|}\sum_{v \in G_n} \hat h\big(C_v(G_n,a^{(0),n},\theta^n)\big)\Big]
\;\xrightarrow[n\to\infty]\;
\E_{\Q}\big[\hat h(G,a^{(0)},\theta)\big],
\quad \text{for all } \hat h \in C_b\big(\mathcal G_\ast[\mathcal A\times\mathcal{A}]\big).
\ee
Recall that $B$ is defined in \eqref{eq:B}. Analogously to \eqref{eq:Picard-iterates}, for $n\in\N$, define the Picard iterates on $G_n$ given by 
\be\label{eq:Picard-iterates-Gn}
(G_n,(a^{(k),n}_v)_{v\in G_n},(\theta^n_v)_{v\in G_n}):=B\big(G_n,(a^{(k-1),n}_v)_{v\in G_n},(\theta^n_v)_{v\in G_n}\big), \quad k\geq 1.
\ee
Then, by Lemma \ref{lemma:B-continuous} and \eqref{eq:Picard-iterates}, \eqref{eq:conv-a0-theta-marks}, \eqref{eq:Picard-iterates-Gn}, given $\hat g \in C_b\big(\mathcal G_\ast[\mathcal A\times\mathcal{A}]\big)$ and $k\in\N_0$, setting $\smash{\hat h:=\hat g\circ B^k}$ in \eqref{eq:conv-a0-theta-marks} yields
\be\label{eq:conv-ak-theta-marks}
\E_{\Q}\Big[\frac{1}{|G_n|}\sum_{v \in G_n} \hat g\big(C_v(G_n,a^{(k),n},\theta^n)\big)\Big]
\;\xrightarrow[n\to\infty]\;
\E_{\Q}\big[\hat g(G,a^{(k)},\theta)\big],
\quad \text{for all } \hat g \in C_b\big(\mathcal G_\ast[\mathcal A\times\mathcal{A}]\big).
\ee
Now given $g \in C_b\big(\mathcal G_\ast[\mathcal A]\big)$, and letting $p:\mathcal G_\ast[\mathcal A\times\mathcal{A}]\to \mathcal G_\ast[\mathcal A]$ denote the projection of the mark onto the first coordinate (which is continuous), setting $\hat g:=g\circ p$ implies  
\be\label{eq:conv-ak-marks}
\E_{\Q}\Big[\frac{1}{|G_n|}\sum_{v \in G_n} g\big(C_v(G_n,a^{(k),n})\big)\Big]
\;\xrightarrow[n\to\infty]\;
\E_{\Q}\big[g(G,a^{(k)})\big],
\quad \text{for all } g \in C_b\big(\mathcal G_\ast[\mathcal A]\big),
\ee
i.e.,~the sequence of $k$-th Picard iterates $\{(G_n,a^{(k),n})\}_{n\in\N}$ corresponding to $(G_n,\theta^n)$ converges in distribution in the local weak sense to the $k$-th Picard iterate $(G,a^{(k)})$ corresponding to $(G,\theta)$. Equivalently, by Definition~\ref{def:local-weak-convergence}, given an independent uniform $\Q$-random variable $U$ on the vertex set of $G_n$, it holds that $\{C_U(G_n,a^{(k),n})\}_{n\in\N}$ converges in distribution to $(G,a^{(k)})$ in $\mathcal{G}_\ast[\mathcal{A}]$.

Now, recalling Assumptions~\ref{assum:U} and \ref{assum:Av}, given any $\eps>0$, choose $k$ so large that $\rho^kM\leq \eps$.
Then, using \eqref{eq:geometric-bound} applied to $G_n$ and $G$ and recalling the definition of $d_\ast$ in \eqref{eq:d-ast-marked}, it holds that
\be
d_\ast\big(C_U(G_n,\bar a^n),C_U(G_n, a^{(k),n})\big)\leq \eps,\quad d_\ast\big((G,\bar a),(G, a^{(k)})\big) \leq \eps,\quad \Q\text{-a.s.~for all }n\in\N.
\ee
Therefore, we can apply Lemma \ref{lemma:aux} on $(\mathcal G_\ast[\mathcal{A}],d_\ast)$ to deduce that $\{C_U(G_n,\bar a^n)\}_{n\in\N}$ converges in distribution to $(G,\bar a)$ in $\mathcal{G}_\ast[\mathcal{A}]$. This is equivalent to convergence of $\{(G_n,\bar a^n)\}_{n\in\N}$ to $(G,\bar a)$ in distribution in the local weak sense, completing the proof of (i).

(ii) Second, assume that $\{(G_n,\theta^n)\}_{n\in\N}$ converges in $\Q$-probability in the local weak sense to $(G,\theta)$. By Remark~\ref{rem:local-weak-convergence-marked}, this means that 
\be\label{eq:prob-conv-theta-marks}
\frac{1}{|G_n|}\sum_{v \in G_n} h\big(C_v(G_n,\theta^n)\big)\;\xrightarrow[n\to\infty]{\ \text{prob.}\ }\;\E_{\Q}\big[h(G,\theta)\big],\quad \text{for all } h \in C_b(\mathcal G_\ast[\mathcal{A}]).
\ee
By the same line of argument as used in \eqref{eq:conv-a0-theta-marks}, \eqref{eq:conv-ak-theta-marks}, \eqref{eq:conv-ak-marks}, we obtain for any $k\in\N_0$ that
\be\label{eq:prob-conv-ak-marks}
\frac{1}{|G_n|}\sum_{v \in G_n} g\big(C_v(G_n,a^{(k),n})\big)
\;\xrightarrow[n\to\infty]{\ \text{prob.}\ }\;
\E_{\Q}\big[g(G,a^{(k)})\big],
\quad \text{for all } g \in C_b\big(\mathcal G_\ast[\mathcal A]\big),
\ee
i.e.,~the sequence of $k$-th Picard iterates $\{(G_n,a^{(k),n})\}_{n\in\N}$ corresponding to $(G_n,\theta^n)$ converges in $\Q$-probability in the local weak sense to the $k$-th Picard iterate $(G,a^{(k)})$ corresponding to $(G,\theta)$. 

For $n\in\N$, given $G_n$, let $U_1^n$, $U_2^n$ be two independent random variables that are uniformly distributed on the vertex set of $G_n$ and set $\smash{C_i^{(k),n}:=C_{U_i^n}(G_n,a^{(k),n})}$ for $i=1,2$. Then, recalling that $|G_n|\to\infty$ in $\Q$-probability, by Lemma~2.8 in \cite{lacker2023local}, \eqref{eq:prob-conv-ak-marks} is equivalent to 
\be\label{eq:prob-conv-product}
\E_\Q[f_1(C_1^{(k),n})f_2(C_2^{(k),n})]\xrightarrow[n\to\infty]{}
\E_\Q[f_1(G,a^{(k)})]\ \E_\Q[f_2(G,a^{(k)})],
\quad \text{for all }f_1,f_2\in C_b(\mathcal{G}_\ast[\mathcal{A}]).
\ee
Since $\mathcal{A}$ is complete and separable, we may equivalently replace $C_b(\mathcal{G}_\ast[\mathcal{A}])$ in \eqref{eq:prob-conv-product} by the space of bounded Lipschitz functions with norm bounded by 1,
\be\label{eq:BL1}
\operatorname{BL}_1(\mathcal{G}_\ast[\mathcal{A}]):=\{f:\mathcal{G}_\ast[\mathcal{A}]\to\R \mid \|f\|_{\text{BL}}\leq 1\},
\ee
where $\|\cdot\|_{\text{BL}}$ is defined as in \eqref{eq:BL-norm} (see \cite{lacker2023local}, proof of Theorem~3.6). That is, 
\be\label{eq:prob-conv-product-Lip}
\E_\Q[f_1(C_1^{(k),n})f_2(C_2^{(k),n})]\xrightarrow[n\to\infty]{}
\E_\Q[f_1(G,a^{(k)})]\ \E_\Q[f_2(G,a^{(k)})],
\quad \text{for all }f_1,f_2\in \operatorname{BL}_1(\mathcal{G}_\ast[\mathcal{A}]).
\ee
Fix $\eps>0$ and $f_1,f_2\in \operatorname{BL}_1(\mathcal{G}_\ast[\mathcal{A}])$. Choose $k$ so large that $\rho^kM\leq \eps$. For $n\in\N$ and $i=1,2$, define in addition
$\bar C_i^n:=C_{U_i^n}(G_n,\bar a^n)$.
Then, by the definition of $d_\ast$ in \eqref{eq:d-ast-marked} and \eqref{eq:geometric-bound} applied to $G_n$, we have that
\be\label{eq:dstar-close-n}
d_\ast(\bar C_i^n,C_i^{(k),n})\leq \eps, \quad \Q\text{-a.s.~for }n\in\N,\ i=1,2.
\ee
Moreover, by \eqref{eq:geometric-bound} applied to $G$, it holds that
\be\label{eq:dstar-close-limit}
d_\ast\big((G,\bar a),(G,a^{(k)})\big)\leq \eps,\quad \Q\text{-a.s.}
\ee
Since $f_1,f_2\in \operatorname{BL}_1(\mathcal{G}_\ast[\mathcal{A}])$ are Lipschitz with respect to $d_\ast$ from \eqref{eq:d-ast-marked} with $\|f_j\|_\infty\leq 1$ and $\operatorname{Lip}(f_j)\leq 1$ for $j=1,2$, we have, using \eqref{eq:dstar-close-n}, that
\be\begin{aligned}\label{eq:term1}
&\Big|\E_\Q\big[f_1(\bar C_1^n)f_2(\bar C_2^n)\big]-\E_\Q\big[f_1(C_1^{(k),n})f_2(C_2^{(k),n})\big]\Big|\\
&\le \E_\Q\Big[\big|f_1(\bar C_1^n)-f_1(C_1^{(k),n})\big|\cdot |f_2( \bar C_2^{n})|\Big]
     +\E_\Q\Big[|f_1(C_1^{(k),n})|\cdot \big|f_2(\bar C_2^n)-f_2(C_2^{(k),n})\big|\Big]\\
&\le \E_\Q\big[d_\ast(\bar C_1^n,C_1^{(k),n})\big]+\E_\Q\big[d_\ast(\bar C_2^n,C_2^{(k),n})\big]\\
&\le 2\eps.
\end{aligned}\ee
Moreover, since $f_1,f_2\in \operatorname{BL}_1(\mathcal{G}_\ast[\mathcal{A}])$ are Lipschitz with respect to $d_\ast$ from \eqref{eq:d-ast-marked} with $\operatorname{Lip}(f_j)\leq 1$ for $j=1,2$, and by \eqref{eq:dstar-close-limit}, it follows that
\be\label{eq:marginal-close}
\Big|\E_\Q\big[f_j(G,a^{(k)})\big]-\E_\Q\big[f_j(G,\bar a)\big]\Big|
\le \E_\Q\big[d_\ast\big((G,a^{(k)}),(G,\bar a)\big)\big]\le \eps,\quad j=1,2.
\ee
Thus, by \eqref{eq:marginal-close} and since $\|f_j\|_\infty\leq 1$ for $j=1,2$,
\be\begin{aligned}\label{eq:term3}
&\Big|\E_\Q[f_1(G,a^{(k)})]\E_\Q[f_2(G,a^{(k)})]-\E_\Q[f_1(G,\bar a)]\E_\Q[f_2(G,\bar a)]\Big|\\
&\le \Big|\E_\Q[f_1(G,a^{(k)})]-\E_\Q[f_1(G,\bar a)]\Big|\cdot \Big|\E_\Q[f_2(G,a^{(k)})]\Big|\\
&\quad +\Big|\E_\Q[f_1(G,\bar a)]\Big|\cdot\Big|\E_\Q[f_2(G,a^{(k)})]-\E_\Q[f_2(G,\bar a)]\Big|\\
&\le 2\eps.
\end{aligned}\ee
Combining \eqref{eq:prob-conv-product-Lip}, \eqref{eq:term1}, and \eqref{eq:term3} with the triangle inequality yields
\[
\limsup_{n\to\infty}
\Big|\E_\Q\big[f_1(\bar C_1^n)f_2(\bar C_2^n)\big]-\E_\Q[f_1(G,\bar a)]\ \E_\Q[f_2(G,\bar a)]\Big|
\le 4\eps.
\]
Since $\eps>0$ was arbitrary, we obtain
\be\label{eq:prob-conv-product-equil}
\E_\Q\big[f_1(\bar C_1^n)f_2(\bar C_2^n)\big]\xrightarrow[n\to\infty]{}
\E_\Q[f_1(G,\bar a)]\ \E_\Q[f_2(G,\bar a)],
\quad \text{for all }f_1,f_2\in \operatorname{BL}_1(\mathcal{G}_\ast[\mathcal{A}]).
\ee
Again, by Lemma~2.8 in \cite{lacker2023local} and the equivalence between testing against $\operatorname{BL}_1(\mathcal G_\ast[\mathcal{A}])$ and against $C_b(\mathcal G_\ast[\mathcal{A}])$, \eqref{eq:prob-conv-product-equil} is equivalent to
\[
\frac{1}{|G_n|}\sum_{v \in G_n} g\big(C_v(G_n,\bar a^n)\big)\;\xrightarrow[n\to\infty]{\ \text{prob.}\ }\;\E_{\Q}\big[g(G,\bar a)\big],
\quad \text{for all } g \in C_b(\mathcal G_\ast[\mathcal{A}]),
\]
i.e.,~$\{(G_n,\bar a^n)\}_{n\in\N}$ converges in $\Q$-probability in the local weak sense to $(G,\bar a)$. This completes the proof of (ii), and hence of Theorem~\ref{thm:convergence}.
\end{proof}

We conclude this section by proving Corollary~\ref{cor:convergence}.

\begin{proof}[Proof of Corollary~\ref{cor:convergence}]
Recall Convention~\ref{convention:quenched} and suppose that the assumptions of Corollary~\ref{cor:convergence} hold.
Consider a hyperfinite unimodular random graph $(G,\theta)=(G,o,\theta)\in\mathcal G_\ast[\mathcal{A}]$ with finitary exhaustion $(H_{\xi_n})_{n\in\N}$ and marks $(\theta_v)_{v\in G}\in\mathcal{A}^G$. 

For a graph $H$ with root $o$, recall that $[[C_o(H)]]$ denotes the unrooted connected component of the root. By Corollary~\ref{cor:sample-unrooted}, the sequence $\{G_n\}_{n\in\N}$ given by $G_n:=[[C_o(H_{\xi_n})]]$ converges in distribution in the local weak sense to $G$. Now, since the marks $\theta^n=\theta_{G_n}\in\mathcal{A}^{G_n}$ are restrictions of the marks $\theta\in\mathcal{A}^G$ to $G_n$, it follows from Remark~\ref{rem:sample-marked} that  $\{(G_n,\theta^n)\}_{n\in\N}$ converges to $(G,\theta)$ in distribution in the local weak sense as well. Denote by $\bar a\in\mathcal{A}^G$ the unique Nash equilibrium of $\mathbb{G}(\mathcal{A}_{ad}^G,U,\theta,G)$ and by $\bar a^n\in\mathcal{A}^{G_n}$ the unique Nash equilibrium of $\smash{\mathbb{G}(\mathcal{A}_{ad}^{G_n},U,\theta^n,G_n)}$ for $n\in\N$.  Then, applying Theorem~\ref{thm:convergence} yields convergence of $\{(G_n,\bar a^n)\}_{n\in\N}$ to $(G,\bar a)$ in distribution in the local weak sense.
\end{proof}

\appendix
\section{Additional Results}
The following result provides sufficient conditions on a filtered probability space ensuring that the associated Hilbert space of square-integrable, progressively measurable processes is separable. Related arguments appear in various forms in the literature, however, we could not locate a statement matching our framework exactly, so we include a self-contained proof for completeness. For simplicity, we omit the tilde in the notation for elements in $\mathcal{A}$ from \eqref{eq:A} from now on.

\begin{lemma}\label{lemma:separable-A}
Let $(\Omega,\mathcal F,\F:=(\mathcal F_t)_{t\in[0,T]},\mathbb P)$ be a filtered probability space with left-continuous filtration such that $\mathcal F_t$ is separable for every $t\in[0,T]$ (in the sense of Definition~\ref{def:separable}). Then the Hilbert space of square-integrable, progressively measurable processes
\be
\mathcal A=\left\{a:\Omega\times [0,T]\to\R \, \Big| \, a \textrm{ is } \mathbb{F}\textrm{-prog.~measurable, } \,\|a\|^2_{\mathcal{A}}:=\E_{\P}\Big[\int_0^T  a(t)^2 dt  \Big]  <\infty \right\}
\ee
is separable.
\end{lemma}

\begin{proof} Let $\|\cdot\|_{L^2}$ denote the norm on $L^2(\Omega\times[0,T],\mathcal F\otimes\mathcal B([0,T]),\P\otimes dt)$ that agrees with $\|\cdot\|_{\mathcal{A}}$ on $\mathcal{A}\subset L^2(\Omega\times[0,T])$. Fix $a\in\mathcal A$ and $\eps>0$. Define the bounded truncations $a^m\in\mathcal{A}$,
\be\label{eq:bounded-truncations}
a^m:=(-m)\vee a\wedge m,\quad m\in\N.
\ee
Let $\{\pi_n\}_{n\in\N}$ denote the sequence of dyadic partitions $\pi_n=(q_j^{n})_{j=0}^{2^n}$ of $[0,T]$ given by 
\be\label{eq:partition}
0=q_0^{n}<q_1^{n}<\ldots<q_{2^n}^{n}=T,\quad q_j^n:=\frac{jT}{2^n}\text{ for }j=0,1,\ldots,2^n.
\ee
For $m,n\in\N$, define the corresponding step process
\be\label{eq:a^(m,n)}
a^{m,n}_t
:=\sum_{j=0}^{2^n-1} c^{m,n}_j\,\mathds 1_{(q_j^{n},q_{j+1}^{n}]}(t),
\quad
c^{m,n}_j
:=\E_{\P}\!\bigg[\frac{1}{q_{j+1}^{n}-q_j^{n}}\int_{q_j^{n}}^{q_{j+1}^{n}} a^m_s\,ds\ \Big|\ \mathcal F_{q_j^{n}}\bigg].
\ee
Then $c^{m,n}_j$ is $\mathcal F_{q_j^{n}}$-measurable with $|c^{m,n}_j|\le m$. Since $t\in(q_j^{n},q_{j+1}^{n}]$
implies $\mathcal F_{q_j^{n}}\subset\mathcal F_t$, the process $a^{m,n}$ is adapted, so being piecewise constant in $t$, it is
progressively measurable, hence $a^{m,n}\in\mathcal A$. Moreover, define $b^{m,n}\in L^2(\Omega\times[0,T])$ by
\be\label{eq:b^(m,n)}
b^{m,n}_t
:=\sum_{j=0}^{2^n-1}\Big(\frac{1}{q_{j+1}^{n}-q_j^{n}}\int_{q_j^{n}}^{q_{j+1}^{n}} a^m_s\,ds\Big)\,
\mathds 1_{(q_j^{n},q_{j+1}^{n}]}(t),
\ee
so that $a^{m,n}_t=\E_{\P}[b^{m,n}_t\mid \mathcal F_{q_j^{n}}]$ for $t\in(q_j^{n},q_{j+1}^{n}]$.

\noindent\emph{Claim 1:} For fixed $m\in\N$, $b^{m,n}\to a^m$ in $L^2$ as $n\to\infty$.

\noindent\emph{Proof of Claim 1:}
Indeed, recalling \eqref{eq:bounded-truncations} and \eqref{eq:b^(m,n)}, for a.e.~$\omega$, the path $t\mapsto a^m_t(\omega)$ lies in $L^2([0,T])$ and $b^{m,n}(\omega)$ is the step function on $\pi_n$ averaging over its subintervals. Therefore, \eqref{eq:partition} implies
\[
\int_0^T |a^m_t(\omega)-b^{m,n}_t(\omega)|^2\,dt\to 0, \quad\text{for a.e.\ }\omega.
\]
Moreover $|a^m-b^{m,n}|\le 2m$, hence dominated convergence yields $\|a^m-b^{m,n}\|_{L^2}\to 0$.

\noindent\emph{Claim 2:} For all $m,n\in\N$,
\be\label{eq:claim2}
\|a^m-a^{m,n}\|_{\mathcal A}^2
\le
2\int_0^T \E_{\P}\Big[\,\big|a^m_t-\E_{\P}[a^m_t\mid\mathcal F_{\underline t^{n}}]\big|^2\,\Big]\,dt
\;+\;
2\|a^m-b^{m,n}\|_{L^2}^2,
\ee
where $\underline t^{n}:=\max\{q^n_j\in\pi_n:\ q^n_j<t\}$ for $t\in(0,T]$, and $\underline 0^{n}:=0$.

\noindent\emph{Proof of Claim 2:}
Fix $t\in(0,T]$ and let $j=j(t)$ be such that $t\in(q^n_j,q^n_{j+1}]$, so that $\underline t^{n}=q^n_j$.
By \eqref{eq:a^(m,n)} and \eqref{eq:b^(m,n)}, we have
$a^{m,n}_t=\E_{\P}[b^{m,n}_t\mid\mathcal F_{q^n_j}]$.
Hence
\[
a^m_t-a^{m,n}_t
=
\big(a^m_t-\E_{\P}[a^m_t\mid\mathcal F_{q^n_j}]\big)
+
\E_{\P}[a^m_t-b^{m,n}_t\mid\mathcal F_{q^n_j}].
\]
Using $(x+y)^2\le 2x^2+2y^2$ and the conditional Jensen inequality, we obtain
\begin{align*}
\E_{\P}\big[|a^m_t-a^{m,n}_t|^2\big]
&\le
2\,\E_{\P}\Big[\,\big|a^m_t-\E_{\P}[a^m_t\mid\mathcal F_{q^n_j}]\big|^2\,\Big]
+
2\,\E_{\P}\Big[\,\big|\E_{\P}[a^m_t-b^{m,n}_t\mid\mathcal F_{q^n_j}]\big|^2\,\Big]\\
&\le
2\,\E_{\P}\Big[\,\big|a^m_t-\E_{\P}[a^m_t\mid\mathcal F_{q^n_j}]\big|^2\,\Big]
+
2\,\E_{\P}\big[|a^m_t-b^{m,n}_t|^2\big].
\end{align*}
Integrating over $t\in[0,T]$ and noting that $q^n_j=\underline t^{n}$ on $(q^n_j,q^n_{j+1}]$ yields \eqref{eq:claim2}.

\noindent\emph{Claim 3:} Fix $m\in\N$. Then
\be\label{eq:claim3}
\int_0^T \E_{\P}\Big[\,\big|a^m_t-\E_{\P}[a^m_t\mid\mathcal F_{\underline t^{n}}]\big|^2\,\Big]\,dt
\;\xrightarrow[n\to\infty]{}\;0.
\ee

\noindent\emph{Proof of Claim 3:}
Fix $t\in(0,T]$. Since $(\pi_n)_n$ is dyadic, the set $\{\underline t^{n}:n\in\N\}$ is a subset of $[0,t)$ with $\underline t^{n}\uparrow t$.
Hence, by monotonicity of the filtration,
\[
\bigvee_{n\in\N}\mathcal F_{\underline t^{n}}
=
\sigma\Big(\bigcup_{n\in\N}\mathcal F_{\underline t^{n}}\Big)
=
\sigma\Big(\bigcup_{s<t}\mathcal F_s\Big)
=\mathcal F_{t-}.
\]
By left-continuity, $\mathcal F_{t-}=\mathcal F_t$.
Since $a^m$ is progressively measurable, $a^m_t$ is $\mathcal F_t$-measurable and $|a^m_t|\le m$, hence $a^m_t\in L^2(\Omega,\mathcal{F},\P)$.
By Corollary 2.4 in Chapter II of \cite{revuz1999continuous},
\[
\E_{\P}[a^m_t\mid\mathcal F_{\underline t^{n}}]\xrightarrow[n\to\infty]{a.s.} \E_{\P}[a^m_t\mid\mathcal F_t]=a^m_t.
\]
Moreover, since $|a^m_t|\le m$ implies $|\E_{\P}[a^m_t\mid\mathcal F_{\underline t^{n}}]|\le m$ for all $n\in\N$, we have
\be\label{eq:dom-conv-bound}
\big|a^m_t-\E_{\P}[a^m_t\mid\mathcal F_{\underline t^{n}}]\big|^2\le 4m^2,
\quad t\in[0,T],\ n\in\N.
\ee
so that dominated convergence on $\Omega$ implies
$$
\E_{\P}\big[|a^m_t-\E_{\P}[a^m_t\mid\mathcal F_{\underline t^{n}}]|^2\big]\xrightarrow[n\to\infty]{}0,\quad t\in(0,T].
$$
Moreover, by \eqref{eq:dom-conv-bound} we have
\[
\E_{\P}\big[\,\big|a^m_t-\E_{\P}[a^m_t\mid\mathcal F_{\underline t^{n}}]\big|^2\,\big]\le 4m^2,
\quad t\in[0,T],\ n\in\N.
\]
Hence dominated convergence on $[0,T]$ implies \eqref{eq:claim3}.

By Claims 1--3, for each fixed $m$ we have $\|a^m-a^{m,n}\|_{\mathcal A}\to 0$ as $n\to\infty$.
Choose $m$ so large that $\|a-a^m\|_{\mathcal A}<\eps/4$, and then choose $n$ so large that
$\|a^m-a^{m,n}\|_{\mathcal A}<\eps/4$. Then, by the triangle inequality,
\be\label{eq:eps/2}
\|a-a^{m,n}\|_{\mathcal A}
\le \|a-a^m\|_{\mathcal A}+\|a^m-a^{m,n}\|_{\mathcal A}
<\frac{\eps}{2}.
\ee
Next, we need to approximate the random coefficients of the step processes in \eqref{eq:a^(m,n)}. For each $n\in\N$ and $j=0,1,\dots,2^n-1$, separability of $\mathcal F_{q_j^n}$
implies that $L^2(\Omega,\mathcal F_{q_j^n},\P)$ is separable, hence we may fix a countable dense subset
\[
D_{j}^{n}\subset L^2(\Omega,\mathcal F_{q_j^n},\P).
\]
Let $\mathcal S\subset\mathcal A$ be the set of all dyadic step processes of the form
\[
s_t=\sum_{j=0}^{2^n-1}\zeta_j\,\mathds 1_{(q_j^n,q_{j+1}^n]}(t),
\quad \zeta_j\in D_{j}^{n},\ n\in\N.
\]
Since for each fixed $n$ the set $\prod_{j=0}^{2^n-1} D_j^n$ is a finite product of countable sets (hence countable), and $\mathcal S$ is a countable union over $n\in\N$ of such sets, it follows that $\mathcal S$ is countable. Finally, to show that $\mathcal S$ is dense in $\mathcal{A}$, write $a^{m,n}$ as in \eqref{eq:a^(m,n)}:
\be\label{eq:a^(m,n)-def}
a^{m,n}_t=\sum_{j=0}^{2^n-1}c^{m,n}_j\,\mathds 1_{(q_j^{n},q_{j+1}^{n}]}(t),
\quad c^{m,n}_j\in L^2(\Omega,\mathcal F_{q_j^{n}},\P).
\ee
For each $j$, choose $\zeta_j\in D_j^n$ such that
\be\label{eq:eps-T-N}
\E_{\P}\big[|c^{m,n}_j-\zeta_j|^2\big]<\frac{\eps^2}{4T 2^n}.
\ee
Define $s\in\mathcal S$ by
\be\label{eq:s}
s_t:=\sum_{j=0}^{2^n-1}\zeta_j\,\mathds 1_{(q_j^{n},q_{j+1}^{n}]}(t).
\ee
Then by \eqref{eq:a^(m,n)-def}, \eqref{eq:eps-T-N}, and \eqref{eq:s},
\be\begin{aligned}
\|a^{m,n}-s\|_{\mathcal A}^2
&=\sum_{j=0}^{2^n-1}(q_{j+1}^{n}-q_j^{n})\,\E_{\P}\big[|c^{m,n}_j-\zeta_j|^2\big]\\
&\le \sum_{j=0}^{2^n-1}T\cdot \frac{\eps^2}{4T 2^n}=\frac{\eps^2}{4},
\end{aligned}\ee
so $\|a^{m,n}-s\|_{\mathcal A}<\eps/2$. Combining this with \eqref{eq:eps/2} and applying the triangle inequality yields
\[
\|a-s\|_{\mathcal A}\le \|a-a^{m,n}\|_{\mathcal A}+\|a^{m,n}-s\|_{\mathcal A}<\eps.
\]
Therefore, $\mathcal S$ is a countable dense subset of $\mathcal A$, so $\mathcal A$ is separable.
\end{proof}


\end{document}